\documentclass[11pt,a4paper,twoside]{article}

\usepackage[T1]{fontenc}
\usepackage[utf8]{inputenc}

\usepackage{amsmath}
\usepackage{mathtools}
\usepackage{bm}
\usepackage{amsthm}

\usepackage{newtxtext}
\usepackage{newtxmath}

\usepackage[
  a4paper,
  margin=28mm,
  headheight=14pt
]{geometry}

\usepackage{microtype}
\usepackage{setspace}
\usepackage{graphicx}
\usepackage{booktabs}
\usepackage{siunitx}
\usepackage{float}
\usepackage{tabularx}
\usepackage{array}
\usepackage{adjustbox}
\usepackage{subcaption}
\usepackage[ruled,vlined]{algorithm2e}

\usepackage{tikz}
\usetikzlibrary{3d}
\usetikzlibrary{arrows.meta}
\usepackage{wrapfig}

\usepackage[
  font=small,
  labelfont=bf,
  labelsep=period
]{caption}

\usepackage[numbers,square,sort&compress]{natbib}
\newtheorem{theorem}{Theorem}[section]     % Numbered by section (e.g., Theorem 1.1)

\newtheorem{lemma}[theorem]{Lemma}         % Now it says "Lemma", not "Theorem"
\newtheorem{problem}[theorem]{Problem}
\newtheorem{corollary}[theorem]{Corollary}

\theoremstyle{definition}

\newtheorem{remark}[theorem]{Remark}

\usepackage{titlesec}

\titleformat{\section}
  {\normalfont\large\bfseries}
  {\thesection.}
  {0.6em}
  {}

\titleformat{\subsection}
  {\normalfont\normalsize\bfseries}
  {\thesubsection.}
  {0.6em}
  {}

\titleformat{\subsubsection}
  {\normalfont\normalsize\itshape}
  {\thesubsubsection.}
  {0.6em}
  {}

\titleformat{\paragraph}
  {\normalfont\normalsize\bfseries}
  {\theparagraph}
  {1em}
  {}

\titlespacing*{\paragraph}
  {0pt}
  {3.25ex plus 1ex minus .2ex}
  {1.5ex plus .2ex}
  
\usepackage{fancyhdr}

\newcommand{\shorttitle}{M.F.P. ten Eikelder and A. Brunk}%Symmetric structure-preserving discretization of N-phase incompressible fluid mixtures with arbitrary density ratios
\newcommand{\shortauthors}{M.F.P. ten Eikelder and A. Brunk}

\fancypagestyle{plain}{%
  \fancyhf{}%
  \fancyfoot[C]{\thepage}%
}
\usepackage{authblk}

\usepackage{xcolor}
\definecolor{darkgreen}{rgb}{0, 0.7, 0}
\definecolor{orange}{rgb}{0.98, 0.6, 0.01}
\definecolor{napiergreen}{rgb}{0.16, 0.5, 0.0}

\usepackage{pifont}

\usepackage{accents}
\newlength{\dhatheight}

\newcommand{\R}{\mathbb{R}}

\numberwithin{equation}{section}
\allowdisplaybreaks

\def\div{\operatorname{div}}

\def\Th{\mathcal{T}_h}
\def\Itau{\mathcal{I}_\tau}

\def\la{\langle}
\def\ra{\rangle}

\def\Th{\mathcal{T}_h}
\def\Vh{\mathcal{V}_h}
\def\Qh{\mathcal{Q}_h}
\def\Xh{\mathcal{X}_h}

\def\nn{\nonumber}

\usepackage{mathtools}

\DeclarePairedDelimiter{\norm}{\|}{\|}
\DeclarePairedDelimiter{\snorm}{|}{|}

\def\w{\mathbf{w}}
\def\vv{\mathbf{v}}

\def\dt{\partial_t}

\def\dtau{d_\tau^{n+1}}

\def\ddt{\frac{\mathrm{d}}{\mathrm{d}t}}

\def\softd{{\leavevmode\setbox1=\hbox{d}%
\hbox to 1.05\wd1{d\kern-0.4ex{\char039}\hss}}}%cstocs

\newcommand{\mA}{\alpha}
\newcommand{\mB}{{\beta}}
\newcommand{\mC}{{\gamma}}

\newcommand{\bJ}{\mathbf{J}}

\newcommand{\bv}{\mathbf{v}}
\newcommand{\trho}{\tilde{\rho}}

\newcommand{\phm}{\phantom{-}}

\usepackage[
  colorlinks=true,
  linkcolor=blue,
  citecolor=blue,
  urlcolor=blue
]{hyperref}

\usepackage[nameinlink,capitalise,noabbrev]{cleveref}

\newcommand{\email}[1]{\href{mailto:#1}{\texttt{#1}}}

\makeatletter
\newcommand{\myref}[1]{\cref{#1}\mynameref{#1}{\csname r@#1\endcsname}}
\newcommand{\Myref}[1]{\Cref{#1}\mynameref{#1}{\csname r@#1\endcsname}}

\newcolumntype{R}[2]{%
    >{\adjustbox{angle=#1,lap=\width-(#2)}\bgroup}%
    l%
    <{\egroup}%
}

\newcommand{\thickhline}{%
    \noalign {\ifnum 0=`}\fi \hrule height 1pt
    \futurelet \reserved@a \@xhline
}

\newcolumntype{"}{@{\hskip\tabcolsep\vrule width 1pt\hskip\tabcolsep}}
\makeatother

\title{\vspace{-1.0cm}
\bfseries
Symmetric structure-preserving discretization of N-phase incompressible fluid mixtures with arbitrary density ratios
}

\author[1]{Marco F.P. ten Eikelder\thanks{\email{marco.eikelder@tu-darmstadt.de}}}
\author[2]{Aaron Brunk\thanks{\email{abrunk@uni-mainz.de}}}

\affil[1]{Institute for Mechanics, Computational Mechanics Group, Technical University of Darmstadt, Germany}
\affil[2]{Institute of Mathematics, Johannes Gutenberg-University Mainz, Germany}

\date{}

\begin{document}

\maketitle
\thispagestyle{plain}

\begin{abstract}
Diffuse-interface models are a widely used framework for interfacial dynamics in complex fluids, in which interfaces are represented through smooth transition layers and capillary effects are encoded by a free-energy functional. For incompressible mixtures with more than two phases, however, robust computation is substantially more difficult because the numerical method should preserve the balance structure of the continuum model, maintain the saturation constraint, dissipate energy, and treat all phases symmetrically even when density ratios are arbitrary. Existing structure-preserving methods are largely developed for binary flows or for formulations that distinguish a reference phase, so a genuinely symmetric N-phase discretization remains lacking. The practical problem is therefore to construct a fully-discrete method for N-phase incompressible Navier–Stokes–Cahn–Hilliard mixture models that retains the key thermodynamic and conservation properties of the continuum equations for arbitrary density ratios.

Here we propose a symmetric fully-discrete method for the N-phase incompressible\\ Navier–Stokes–Cahn–Hilliard mixture model with arbitrary density ratios. The method yields a fully-discrete problem in which every solution satisfies exact phase volume conservation, phase mass conservation, total volume conservation, total mass conservation, and a discrete energy-dissipation law. In addition, if the volume-saturation constraint holds for the initial data, then it is preserved at every time step. We numerically verify these structure-preserving properties and demonstrate the robustness of the method in representative multiphase flow problems. The resulting scheme provides a computational framework for incompressible N-phase mixture flows with complex interfacial dynamics and arbitrary density contrasts.
\end{abstract}

\section{Introduction}\label{sec: Introduction}

$N$-phase incompressible fluid mixtures arise in applications where several immiscible or partially miscible phases interact through moving interfaces. Diffuse-interface models provide a natural description of such systems by replacing sharp material boundaries with smooth transition layers and representing interfacial effects through phase-field energies. The model class considered in this work belongs to the Navier--Stokes--Cahn--Hilliard (NSCH) family, whose binary form is rooted in the hydrodynamic phase-field coupling introduced by \cite{hohenberg1977theory}, now commonly referred to as model H. Early two-phase NSCH formulations were mainly developed for fluids with matched densities, whereas many applications require models that allow density contrasts. This has led to several non-matching-density extensions, including influential contributions by \cite{lowengrub1998quasi,ding2007diffuse,abels2012thermodynamically,abels2024mixture}. Although these formulations pursue the same physical objective, they can differ in their balance-law structure, dissipation statements, and behavior in limiting single-fluid regimes. A continuum-mixture-theory framework has recently clarified these relationships and identified a two-phase NSCH mixture model that is invariant with respect to the choice of fundamental variables \cite{eikelder2023unified}.

The numerical approximation of binary NSCH systems has generated a substantial literature, in particular on methods that retain key structural properties of the continuum equations. For matched densities, energy-stable and structure-preserving schemes have been developed in a variety of settings; see, for example, \cite{feng2006fully,kay2007efficient,chen2016efficient,diegel2017convergence,brunk2023second} for matched density models. The non-matching-density case introduces additional coupling through the density and mass balance, and has motivated further schemes such as those in \cite{khanwale2023projection,guillen2014splitting,garcke2016stable,shen2013mass,giesselmann2015energy,ten2024divergence}. Closest to the present work is the recent fully-discrete method \cite{brunk2026simple}, proposed by the authors, for the two-phase NSCH mixture model with arbitrary density ratios. There, unconditional energy stability is obtained despite the use of a positive density extension in the kinetic energy, as required by phase-field modifications.

For ternary and genuinely multiphase systems, fewer fully-discrete structure-preserving schemes are available. For the model developed by Boyer and co-authors \cite{boyer2014hierarchy}, several discretization schemes have been proposed and applied, mostly for the ternary case $N=3$ \cite{boyer2006study, boyer2010cahn, minjeaud2013unconditionally,LEE20084787}, with a multigrid extension considered in \cite{cms/1250880209}. Later developments include SAV/EIQ schemes \cite{Yang2017, Yang21}. Extensions to $N>3$ are only briefly outlined in \cite{boyer2014hierarchy}.
For the multiphase model developed by \cite{dong2018multiphase}, additional numerical schemes and extensions, such as wall contributions, have been developed and tested for $N\geq 3$ \cite{DONG201721, DONG201598, DONG2014691, ZHANG2020109115}. It should be noted that $N=3$ is a separate case as the number of pairwise interfaces equals the number of phases, whereas for $N>3$ the number of pairwise interfaces is larger than the number of phases.

For $N$-phase flow, diffuse-interface modelling is even more flexible and substantially more challenging. Multiphase NSCH-type models have been proposed in a variety of forms \cite{boyer2014hierarchy,dong2018multiphase,ten2024thermodynamically,Abels2025}. Most recently, a unified mixture-theory framework for $N$-phase NSCH mixture models with a single mixture momentum equation has been established in \cite{ten2025unified}. While such a general $N$-phase framework is now available, its practical use requires constitutive closure, interfacial calibration, and a numerical method that preserves the structural properties of the resulting equations. In \cite{mixtureaware2026}, we derive mixture-aware constitutive closures for $N$-phase NSCH mixture models by imposing reduction consistency under merging physically identical phases, leading to admissible multiphase free energies and mobilities. Furthermore, in \cite{surfacetension2026}, we develop a surface-tension calibration procedure that determines capillarity parameters from prescribed pairwise surface tensions and mesh-resolvable interface widths. These results provide the modeling input for $N$-phase simulations. The remaining question, and the focus of the present work, is how to discretize the resulting $N$-phase equations so that their defining conservation and dissipation structure is retained at the fully-discrete level.

This task is substantially more delicate in the $N$-phase setting than in the binary case. The phase volume fractions are coupled through the saturation constraint, the densities enter both phase and total mass balances, capillarity acts through chemical potentials derived from an $N$-phase free energy, and the flow is governed by a single mixture momentum equation. A useful discretization should therefore do more than remain energy stable. Ideally, it should also preserve phase-wise volume and mass balances, conserve total volume and total mass, and maintain the saturation constraint exactly. If these properties are lost at the discrete level, even small local defects can accumulate into nonphysical drift in composition, density, or total volume, especially over long time integrations and in strongly coupled multiphase regimes.

A further difficulty is symmetry. At the continuum level, the mixture model treats the phases on equal footing, and this symmetry is important both conceptually and practically. In many numerical formulations, however, symmetry is broken early: one phase is eliminated through the saturation constraint, auxiliary variables are introduced relative to a designated reference phase, or the discrete equations are organized in a way that privileges particular labels. Such choices may be convenient algebraically, but they also build a distinction between phases into the discrete formulation. For $N$-phase simulations, it is therefore desirable to develop a discretization that remains symmetric with respect to the phase set while preserving the underlying thermodynamic structure.

The present work addresses this problem. We develop a symmetric structure-preserving discretization for the $N$-phase incompressible NSCH mixture model with arbitrary density ratios. The method is formulated directly in terms of the full set of phase variables and a single mixture velocity, without introducing a distinguished reference phase. Its construction is close in spirit to the recent fully-discrete energy-stable methodology for the two-phase problem \cite{brunk2026simple}, but is adapted here to the genuinely multiphase setting. Its central property is that the fully-discrete problem inherits the key structural identities of the continuum model. In particular, we prove that every discrete solution satisfies exact phase volume conservation, phase mass conservation, total volume conservation, total mass conservation, and a discrete energy-dissipation inequality. Moreover, if the initial data satisfy the volume-saturation constraint, then this constraint is preserved pointwise at all subsequent time steps.

The structure of the remainder of the paper is as follows. In \cref{sec:model} we present the NSCH model with its properties. Next, in \cref{sec:reform} we provide an equivalent alternative formulation amenable for a fully-discrete method. Then, in \cref{sec:scheme} we present the numerical scheme and its structural properties. In \cref{sec:numerics} we demonstrate the structural properties of the scheme numerically. Finally, we close the paper with a conclusion and outlook in \cref{sec:conclusion and outlook}.

\section{The Navier--Stokes--Cahn--Hilliard model}\label{sec:model}

This section is concerned with the Navier--Stokes--Cahn--Hilliard model. Following the presentation of \cite{mixtureaware2026}, \cref{subsec:gov eq} presents the governing equations, and \cref{subsec: struct prop} provides its structural properties. Next, \cref{subsec: Helmholtz free energy} discusses the Helmholtz free energy. The associated surface tension and interface width are provided in \cref{sec:surface_tension} and \cref{sec:interface_width}, respectively. Finally, the equilibrium conditions are stated in \cref{sec:equilibrium_conditions}, and \cref{sec:hierarchy_consistency} provides the mixture-aware properties.  %Finally, \ref{subsec:dim less} presents the non-dimensional formulation.

\subsection{Governing equations}\label{subsec:gov eq}
We consider the $N$-phase Navier--Stokes--Cahn--Hilliard mixture system with non-matching densities given by ten Eikelder \cite{ten2025unified}:
\begin{subequations}\label{eq:sys1}
  \begin{align}\partial_t (\rho \bv) + {\rm div} \left( \rho \bv\otimes \bv \right) + \sum_\mB \phi_\mB \nabla \mu_\mB + \nabla \lambda
    %&\nn\\
    - {\rm div} \boldsymbol{\tau}-\rho\mathbf{b} &=~ 0, \label{eq: intro mass: mom}\\
  \partial_t \phi_\mA  + {\rm div}(\phi_\mA  \bv) +\rho_\mA^{-1}{\rm div} \bJ_\mA &=~0.\label{eq: intro mass: mass}
  \end{align}
\end{subequations}
The system couples the mixture momentum balance \eqref{eq: intro mass: mom} to the $N$-phase mass balances \eqref{eq: intro mass: mass}. The equations are posed in a spatial domain $\Omega \subset {\R}^d$ of dimension $d=2,3$, with boundary $\partial\Omega$ and outward unit normal $\mathbf{n}$. The unknowns are the fluid (mass-averaged) velocity $\bv: \Omega \rightarrow {\R}^d$, the volume fractions $\phi_\mA: \Omega \rightarrow {\R}$, $\mA =1,...,N$, and the Lagrange multiplier pressure $\lambda: \Omega \rightarrow {\R}$. We assume the absence of void spaces so that the volume fractions are subject to the saturation constraint
\begin{align}\label{eq: saturation constraint}
    \sum_\mA \phi_\mA =1.
\end{align}
In addition, the physical values of the volume fractions are $0\leq \phi_\mA\leq 1$.  Hence, the volume fractions lie in the Gibbs simplex $\mathcal{G} = \{\boldsymbol{\phi}\in\mathbb{R}^N: 0 \leq \phi_\mA \leq 1,\sum_\mA \phi_\mA = 1\}$. The system is supplemented with the initial conditions $\bv(\mathbf{x},0) = \bv_0(\mathbf{x})$ and $\phi_\mA(\mathbf{x},0) = \phi_{\mA,0}(\mathbf{x})$. The specific phase densities $\rho_\mA$ are constant, the partial mass densities are defined by $\trho_\mA = \rho_\mA \phi_\mA$, and the total mixture density is $\rho = \sum_\mA \trho_\mA$.
The body force is taken to be gravitational and is represented by $\mathbf{g} = -g \mathbf{j}$, where $g$ is the gravitational constant, $\mathbf{j} = \nabla y$ is the vertical unit vector, and $y$ denotes the vertical coordinate. The viscous stress is given by $\boldsymbol{\tau} = \nu (2 \nabla^s \bv+\bar{\lambda}({\rm div}\bv) \mathbf{I}) $, with dynamic viscosity $\nu$, coefficient $\bar{\lambda} = -2/d$, and symmetric velocity gradient $\nabla^s \bv=(\nabla \bv + (\nabla \bv)^T)/2$. The Helmholtz free energy $\Psi$ associated with \eqref{eq:sys1} is assumed to belong to the constitutive class
\begin{align}
    \Psi = \Psi\left(\left\{\phi_\mA\right\},\left\{\nabla \phi_\mA\right\}\right).
\end{align}
The corresponding chemical potentials are denoted by $\mu_\mA$ and $g_\mA$, with
\begin{align}\label{eq: chem pot}
    \mu_\mA = \dfrac{ \partial \Psi}{\partial \phi_\mA} - {\rm div}\dfrac{\partial \Psi}{\partial\nabla \phi_\mA},
\end{align}
and $g_\mA = \rho_\mA^{-1}(\mu_\mA + \lambda)$. 

%\begin{remark}[Mixture velocity]The system is formulated in terms of the mass-averaged velocity $\vv$. Alternatively, it can be expressed using other mixture velocities; for details on the volume-averaged velocity formulation, we refer to ten Eikelder \cite{ten2025unified}.\end{remark}

Invoking the saturation constraint in the addition of the mass balance equations \eqref{eq: intro mass: mass} provides:
\begin{align}\label{eq: LM 1}
    {\rm div}\bv + \displaystyle\sum_{\mA}  \rho_\mA^{-1} {\rm div} \bJ_\mA = 0,
\end{align}
where the diffusive fluxes are $
  \bJ_\mA =- \sum_\mB M_{\mA\mB}\nabla g_\mB$, with $M_{\mA\mB}$ denoting the degenerate mobility matrix.  We employ the mixture-aware mobility introduced in \cite{mixtureaware2026}, namely
\begin{subequations}
\begin{align}
M_{\alpha\beta}(\boldsymbol{\phi})=&~-m_{\alpha\beta}\phi_\alpha\phi_\beta \quad (\alpha\neq\beta),\\
M_{\alpha\alpha}(\boldsymbol{\phi})=&~\phi_\alpha\sum_{\beta\neq\alpha} m_{\alpha\beta}\phi_\beta.
\end{align}
\end{subequations}
The coefficients $m_{\alpha\beta}$ specify the pairwise mobilities and are chosen according to the constitutive closure described there.
In the case of pure phases ($\phi_\mA = 1$ for some $\mA = 1,...,N$) the mobility tensor is zero, and consequently, the peculiar velocities vanish ($\bJ_\mA = 0$) so that the fluid velocity is divergence-free (${\rm div} \bv = 0$). By treating the chemical potentials as independent unknowns, the model can be written as a closed system of $2N+2$ equations for the $2N+2$ state variables $(\bv, \left\{\phi_\mA\right\}_{\mA=1,...,N}, \lambda, \left\{\mu_\mA\right\}_{\mA=1,...,N})$:
\begin{subequations}\label{eq:sys2}
  \begin{align}\partial_t (\rho \bv) + {\rm div} \left( \rho \bv\otimes \bv \right) + \sum_\mB \phi_\mB \nabla \mu_\mB + \nabla \lambda
    %&\nn\\
    - {\rm div} \boldsymbol{\tau}-\rho\mathbf{b} &=~ 0, \label{eq:sys2: mom}\\
  \partial_t \phi_\mA  + {\rm div}(\phi_\mA  \bv) +\rho_\mA^{-1}{\rm div} (\bJ_\mA )&=~0,\label{eq:sys2: mass}\\
  {\rm div}\bv + \displaystyle\sum_{\mA}  \rho_\mA^{-1} {\rm div} \bJ_\mA &=~ 0, \label{eq:sys2: div}\\
   \mu_\mA - \dfrac{ \partial \Psi}{\partial \phi_\mA} + {\rm div}\dfrac{\partial \Psi}{\partial\nabla \phi_\mA} &~=0.\label{eq:sys2: chem}
  \end{align}
\end{subequations}
The pressure $\lambda$ acts as a Lagrange multiplier that enforces \eqref{eq:sys2: div} (which follows from \eqref{eq: saturation constraint}), see also \cite{ten2025unified}. The system \eqref{eq:sys2} is symmetric in terms of the volume fraction variables.

%\begin{remark}[Saturation constraint]
%  As an alternative, one can directly substitute the saturation constraint \eqref{eq: saturation constraint} to reduce the set of volume fraction state variables to $\left\{\phi_\mA\right\}_{\mA\neq \mB}$ for some fixed $\mB \in \left\{1,...,N\right\}$. This leads to a reduced system in which \eqref{eq:sys2: mom}-\eqref{eq:sys2: chem} constitute $2N+1$ equations for $2N+1$ state variables ($\bv, \left\{\phi_\mA\right\}_{\mA\neq \mB}, \lambda, \left\{\mu_\mA\right\}$). This system is not symmetric in terms of $\phi_\mA$.
%\end{remark}

\subsection{Structural properties}\label{subsec: struct prop}
The continuum system satisfies several structural identities that are central to the construction of the numerical method. Under suitable boundary conditions, the phase masses and volume-fraction integrals are conserved, while the total energy is dissipated according to:
\begin{subequations}    
\begin{align}
   \ddt \int_\Omega \trho_\mA~{\rm d}\Omega =&~ 0, \qquad \mA = 1,...,N,\\
   \ddt \int_\Omega \phi_\mA~{\rm d}\Omega =&~ 0, \qquad \mA = 1,...,N,\\
   \ddt \mathcal{E}\left(\vv,\left\{\phi_\mA\right\}\right) =&~ -\mathcal{D}\left(\vv, \left\{g_\mA\right\}\right)\leq 0, \label{eq: global energy evolution}
\end{align}
\end{subequations}
The energy functional $\mathcal{E}$ and the corresponding dissipation rate $\mathcal{D}$ are defined by
\begin{subequations}
\begin{align}
 \mathcal{E}\left(\vv,\left\{\phi_\mA\right\}\right):=&~\int_\Omega K\left(\vv,\left\{\phi_\mA\right\}\right) + G\left(\left\{\phi_\mA\right\}\right) + \Psi\left(\left\{\phi_\mA\right\},\left\{\nabla \phi_\mA\right\}\right)~{\rm d}\Omega, \label{eq:defEnergy}\\
 K\left(\vv,\left\{\phi_\mA\right\}\right) =&~ \frac{1}{2}\rho\left(\left\{\phi_\mA\right\}\right)\snorm{\vv}^2,\\ G\left(\left\{\phi_\mA\right\}\right) =&~ \rho\left(\left\{\phi_\mA\right\}\right) g y,\\
 \mathcal{D}\left(\vv, \left\{g_\mA\right\}\right):=&~ \displaystyle\int_{\Omega} 2 \nu \left( \nabla^s \bv - \frac{1}{d} ({\rm div} \mathbf{v}) \mathbf{I}\right):\left(\nabla^s \bv - \frac{1}{d} ({\rm div} \mathbf{v}) \mathbf{I}\right)+ \nu \left(\bar{\lambda} + \frac{2}{d}\right)\left({\rm div} \mathbf{v}\right)^2 \nn\\
    &~\quad\quad+ \displaystyle\sum_{\mA,\mB} \left(\nabla g_\mA\right)^T M_{\mA\mB} \nabla g_\mB~{\rm d}\Omega \geq 0.
 \label{eq:defDissipation}
\end{align}
\end{subequations}
Thus, the energy consists of the kinetic contribution $K$, the gravitational contribution $G$, and the Helmholtz free-energy density $\Psi$, while $\mathcal{D}$ collects the viscous and diffusive dissipation mechanisms. In the absence of gravity ($\mathbf{g}=0$), the momentum equation can alternatively be expressed in conservative form by using the Korteweg tensor identity
  \begin{align}
    {\rm div} \left(\left(\sum_\mA \mu_\mA\phi_\mA -\Psi\right)\mathbf{I} + \sum_\mA \nabla \phi_\mA \otimes \dfrac{\partial \Psi}{\partial \nabla \phi_\mA} \right) + \nabla \lambda =  \sum_\mA\trho_\mA\nabla g_\mA. \label{eq: korteweg}
  \end{align}
This conservative form is compatible with angular momentum conservation only when the corresponding Korteweg contribution is symmetric. Hence, the tensor $\sum_\mA \nabla \phi_\mA \otimes \partial \Psi/\partial \nabla \phi_\mA$ must be symmetric, which restricts the thermodynamically admissible choices of the free energy.

\subsection{Helmholtz free energy}\label{subsec: Helmholtz free energy}
The Cahn--Hilliard free energy per unit volume is written as:
\begin{subequations}\label{eq: Cahn Hilliard free energy}
\begin{align}
    \Psi =&~ \Psi_{0}\left(\left\{\phi_\mB\right\}\right) + \Psi_\nabla\left(\left\{\nabla \phi_\mB\right\}\right),\\
    \Psi_\nabla\left(\left\{\nabla \phi_\mB\right\}\right) =&  \frac{\varepsilon_0}{2}\sum_{\mA,\mB} \bar{\kappa}_{\mA\mB} \nabla \phi_\mA \cdot \nabla \phi_\mB,
\end{align}
\end{subequations}
where $\Psi_0$ is the bulk free energy that does not contain any gradient contributions, whereas the gradient term represents an excess free energy in the interfacial region. Here $\bar{\kappa}_{\mA\mB}=\bar{\kappa}_{\mB\mA}$ defines the symmetric (volume-based) capillarity matrix that includes both self- and cross-interaction terms among phases, and $\varepsilon_0$ represents the interface thickness measure. We note that the symmetry assumptions yield a symmetric Korteweg stress, cf. \eqref{eq: korteweg}.  The associated chemical potential takes the form:
\begin{align}
    \mu_\mA = \partial_{\phi_\mA} \Psi_0 - \varepsilon_0 \sum_{\mB} \bar\kappa_{\mA\mB}\Delta \phi_\mB,
\end{align}
where we used the independence of the state variables $\left\{\phi_\mA\right\}$.

With the aid of the saturation constraint \eqref{eq: saturation constraint} we reformulate the gradient contribution as \cite{surfacetension2026}:
    \begin{subequations}
      \begin{align}
     \frac{\varepsilon_0}{2}\sum_{\mA,\mB} \bar{\kappa}_{\mA\mB} \nabla \phi_\mA \cdot \nabla \phi_\mB =&~ -\frac{\varepsilon_0}{2}\sum_{\substack{\mA,\mB\\\mA < \mB}} \bar{\sigma}_{\mA\mB} \nabla \phi_\mA \cdot \nabla \phi_\mB,\\
     \bar{\sigma}_{\mA\mB} =&~ \bar{\kappa}_{\mA\mA}+\bar{\kappa}_{\mB\mB}-2\bar{\kappa}_{\mA\mB} \geq 0,
     \end{align}
     \end{subequations}
where we used that $\boldsymbol{\kappa}$ is positive semidefinite on the tangent space. The (symmetric) capillarity tensor $\bar{\sigma}_{\mA\mB}=\bar{\sigma}_{\mB\mA}$ captures only the interfacial contributions between different phases, excluding self-interactions: $\bar{\sigma}_{\mA\mA} = 0$. The choice $\bar{\kappa}_{\mA\mB}=\tfrac{1}{2}(\bar{\kappa}_{\mA\mA}+\bar{\kappa}_{\mB\mB})$ causes the gradient contribution in \eqref{eq: Cahn Hilliard free energy} to vanish. %, due to the saturation constraint \eqref{eq: saturation constraint}.
The chemical potential may now be written as: 
\begin{align}
    \mu_\mA =&~ \partial_{\phi_\mA} \Psi_0 + \varepsilon_0 \sum_{\substack{\mB\\\mB > \mA}} \bar{\sigma}_{\mA\mB}\Delta \phi_\mB.%\mu_\mA =&~ \partial_{\phi_\mA} \Psi_0 - \tfrac{1}{2}\sum_{\substack{\mB\\\mB < \mA}} \sigma_{\mB\mA}\Delta \phi_\mB- \tfrac{1}{2}\sum_{\substack{\mB\\\mB > \mA}} \sigma_{\mB\mA}\Delta \phi_\mC \nn= ~ \partial_{\phi_\mA} \Psi_0 - \tfrac{1}{2}\sum_{\substack{\mB\\\mB \neq \mA}} \sigma_{\mA\mB}\Delta \phi_\mB.
\end{align}
Note that on the continuous level all these formulations are equivalent by using the saturation condition. For the numerical approximation, we prefer the formulation \eqref{eq: Cahn Hilliard free energy}. Under the assumption that the matrix $\kappa_{\mA\mB}$ is symmetric positive semi-definite, this yields a convex gradient energy, which is more amendable during discretization. We emphasize that the reduced form $\sigma_{\mA\mB}$ without the saturation constraint, does not imply a convex energy functional, as this matrix is not symmetric positive semi-definite.

In the current article we adopt the free energy proposed in \cite{mixtureaware2026}:
\begin{subequations}\label{eq: Flory-Huggins}
\begin{align} 
\Psi_0\left(\left\{\phi_\mB\right\}\right) =&~ \frac{\bar{\Psi}_0}{\varepsilon_0},\\
\bar{\Psi}_0\left(\left\{\phi_\mB\right\}\right) =&~ \bar{W} \left(\sum_{\mA}  \phi_\mA \log \phi_\mA -\sum_{\mA\mC} \chi_{\mA\mC} \phi_\mA \phi_\mC\right),
\end{align}
\end{subequations}
where $\chi_{\mA\mC}=\chi_{\mC\mA}$ are symmetric interaction parameters. Rearrangement of the summation and use of the saturation constraint \eqref{eq: saturation constraint} yields
\begin{align}
\bar{\Psi}_0\left(\left\{\phi_\mB\right\}\right) = \bar{W} \left(\sum_{\mA}  \phi_\mA \log \phi_\mA - \sum_{\mA}\chi_{\mA\mA}\phi_\mA + \sum_{\mA<\mC}\chi_{\mA\mC}^{\mathrm{FH}}\phi_\mA\phi_\mC \right).
\end{align}
In the above form, we can drop the linear term since only gradients of $\partial_{\phi_{\mB}}\Psi_0$ contribute to the system, while the typical parameters of the Flory interactions are given by $\chi_{\mA\mC}^{\mathrm{FH}}:=\chi_{\mA\mA}+\chi_{\mC\mC}-2\chi_{\mA\mC}$. Note that only $\chi_{\mA\mC}^{\mathrm{FH}}$ are the free parameters, while the $\chi_{\mA\mC}$ can be shifted without changing the corresponding $\chi_{\mA\mC}^{\mathrm{FH}}$.

\subsection{Surface tension}
\label{sec:surface_tension}
We briefly discuss the definition of surface tension for the present free-energy class. The details are provided in \cite{surfacetension2026}.

Let $ \mathbf{p}^{(\mA\mB)}$ be a curve connecting two coexistence points $\mathbf{b}^{(\mA)}$ and $\mathbf{b}^{(\mB)}$,
\begin{align}\label{eq: boundary conditions phi}
  \mathbf{p}^{(\mA\mB)}: [0,1] \to \mathcal{G},\qquad
  \mathbf{p}^{(\mA\mB)}(0)=\mathbf{b}^{(\mA)},\quad
  \mathbf{p}^{(\mA\mB)}(1)=\mathbf{b}^{(\mB)},
\end{align}
and let $s:{\R}\to[0,1]$ be a strictly monotone parameterization of this path. The corresponding free energy per unit interfacial area is \cite{surfacetension2026}:
\begin{align}\label{eq: free energy interfacial area phi}
\gamma_{\mA\mB}[\mathbf{p}^{(\mA\mB)}] = 
  \int_0^1
      \sqrt{
      2\bar{\Omega}_{\mathbf{b}^{(\mA)}}(\mathbf p^{(\mA\mB)}(s))
      \sum_{\gamma,\delta}
      \bar{\kappa}_{\gamma\delta}
      p_\gamma^{(\mA\mB)\prime}(s)
      p_\delta^{(\mA\mB)\prime}(s)
      }
   {\rm d}s,
\end{align}
where $\bar{\Omega}_{\mathbf{b}^{(\mA)}}(\mathbf p^{(\mA\mB)}(s))$ is the grand potential of the bulk free energy $\bar{\Psi}_0$ relative to $\mathbf{b}^{(\mA)}$. 
The equilibrium profile minimizes \eqref{eq: free energy interfacial area phi} subject to the boundary conditions \eqref{eq: boundary conditions phi}, and the surface tension between phases $\mA$ and $\mB$ is defined as the minimal free energy per unit interfacial area
\begin{align}\label{eq: ST FH}
  \gamma_{\mA\mB}
  =
  \inf_{\mathbf{p}^{(\mA\mB)}}\gamma_{\mA\mB}[\mathbf{p}^{(\mA\mB)}].
\end{align}
The surface tension \eqref{eq: free energy interfacial area phi} thus depends on both the homogeneous and gradient parts of the free energy. The parameter $\varepsilon_0$ changes the diffuse-interface thickness but does not alter the values of the surface tensions $\gamma_{\mA\mB}$; see also \cite{surfacetension2026}. The paper \cite{surfacetension2026} provides an algorithm for choosing $\bar\kappa_{\gamma\delta}$ so that prescribed pairwise surface tensions are matched.

\subsection{Interface width}
\label{sec:interface_width}

Next, we briefly recall the notion of interface widths for the present free-energy class and refer to \cite{surfacetension2026} for the detailed derivation and discussion. 
%Let us denote the physical one-dimensional profile corresponding to $\mathbf{p}^{(\mA\mB)}$ as $\mathbf p^{(\mA\mB)}(z)=\mathbf{p}^{(\mA\mB)}(s(z))$, where the normal coordinate is $z=\mathbf x\cdot\boldsymbol{\nu}$. The interface width between phases $\mA$ and $\mB$ is then defined as the extent in the normal direction over which the profile passes between the corresponding levels,
The interface width between phases $\mA$ and $\mB$ is defined as the extent in the normal coordinate $z=\mathbf x\cdot\boldsymbol{\nu}$ over which the one-dimensional equilibrium profile passes between two prescribed transition
\begin{align}
\varepsilon_{\mA\mB}(a,b)
:=
|z_b-z_a|,
\label{eq:width_profile_short}
\end{align}
where $z_a$ and $z_b$ are the locations at which the profile reaches the prescribed levels $a$ and $b$. We provide precise numerical values of $a$ and $b$ in \cite{surfacetension2026}. The interface width may also be expressed using the following formula:
\begin{align}
\varepsilon_{\mA\mB}(a,b)
= \varepsilon_0
\int_a^b \sqrt{
\frac{
\mathbf p^{(\mA\mB)\prime}(s)^{\!\top} {\bar{\boldsymbol{\kappa}}}  \mathbf p^{(\mA\mB)\prime}(s)
}{
2 \bar{\Omega}_{\mathbf{b}^{(\mA)}}(\mathbf p^{(\mA\mB)}(s))
}
}
{\rm d}s.
\label{eq:width_energy_short}
\end{align}
The definition depends on the path of integration. 
In \cite{surfacetension2026} we provide a symmetry-consistent choice for the path of integration. Furthermore, we observe that the interface widths are determined by the competition between the capillarity metric $\bar{\boldsymbol{\kappa}}$ and the grand potential $\bar{\Omega}_{\mathbf{b}^{(\mA)}}$, and scale with $\varepsilon_0$. 
We note that there is in general no single diffuse-interface thickness in an $N$-phase model; different coexistence pairs may have different interface widths. We define the practical interface width
\begin{align}\label{eq:representative_width}
    \varepsilon
    :=
    \min_{\mA<\mB}\varepsilon_{\mA\mB}.
\end{align}

\subsection{Equilibrium conditions}
\label{sec:equilibrium_conditions}

The general equilibrium conditions for the $N$-phase NSCH mixture model are discussed in \cite{ten2025unified,mixtureaware2026,surfacetension2026}. In the absence of gravity and with no-slip boundary conditions, one has
\begin{align}
\vv=0,
\qquad
g_\alpha=\text{const},
\qquad
\bJ_\alpha=0
\qquad\text{for all }\alpha.
\end{align}
In \cite{mixtureaware2026,surfacetension2026} it is shown that the equilibrium profiles $\boldsymbol{\phi}^{\mathrm{eq}}$ are given by the constrained minimization problem
\begin{align}
\boldsymbol{\phi}^{\mathrm{eq}}
=
{\arg\inf}_{\boldsymbol{\phi}}
\Biggl\{
\int_\Omega \Psi(\boldsymbol{\phi},\nabla\boldsymbol{\phi}) {\rm d}x
\;:\;
\sum_{\alpha=1}^{N}\phi_\alpha = 1
\ \text{in }\Omega,
\quad
\int_\Omega \phi_\alpha {\rm d}x = m_\alpha
\ \text{for all }\alpha
\Biggr\},
\label{eq:equilibrium_min_problem}
\end{align}
where the constants $m_\alpha$ denote the total volume-fraction integrals.

\subsection{Mixture-aware closure}\label{sec:hierarchy_consistency}

A central feature of the thermodynamic closure derived in \cite{mixtureaware2026} is its \emph{mixture awareness}: the constitutive model remains consistent under changes in the effective phase set. This concerns three properties, of which the first one (reduction under merging of identical phases) is the decisive requirement, in the sense that it uniquely determines the closure class introduced in \cref{subsec: Helmholtz free energy}. The other two properties (reduction on simplex faces when a phase is absent, and invariance of absent phases) are then natural consequences of this mixture-aware structure. We recall these three properties below, since they provide important structural reference points for the discretization developed in this paper.

The first property concerns \emph{merging identical phases}. In applications, it is common that a phase is split into sublabels for modelling or numerical convenience, or that two labels happen to have identical constitutive parameters. In such a situation, the $N$-phase model should reduce exactly to an $(N\!-\!1)$-phase model after merging the corresponding volume fractions.

\begin{theorem}[Reduction by merging identical phases {\cite{mixtureaware2026}}]\label{thm:hier_merge}
Assume $N\ge 3$ and consider two phases, without loss of generality phases $1$ and $N$, which are identical in the sense that
\begin{subequations}
\begin{align}
\rho_1&=\rho_N,\\
\chi_{1\mA}&=\chi_{N\mA},\qquad \mA=1,\dots,N,\\
\kappa_{1\mA}&=\kappa_{N\mA},\qquad \mA=1,\dots,N,\\
m_{1\mA}&=m_{N\mA},\qquad \mA=2,\dots,N-1.
\end{align}
\end{subequations}
Define the merged variables by
\begin{subequations}
\begin{align}
\widehat\phi_1&:=\phi_1+\phi_N,\\
\widehat\phi_\mA&:=\phi_\mA,\qquad \mA=2,\dots,N-1.
\end{align}
\end{subequations}
Then the $N$-phase NSCH system reduces to the $(N-1)$-phase NSCH system for
\begin{align}
\widehat{\boldsymbol\phi}=(\widehat\phi_1,\widehat\phi_2,\dots,\widehat\phi_{N-1}),
\end{align}
with the reduced coefficients:
\begin{subequations}
\begin{align}
\hat\rho_1&:=\rho_1,\qquad \hat\rho_\mA:=\rho_\mA,\qquad \mA=2,\dots,N-1,\\
\hat\chi_{11}&:=\chi_{11}=\chi_{1N}=\chi_{NN},\\
\hat\chi_{1\mA}&:=\chi_{1\mA}=\chi_{N\mA},\qquad \mA=2,\dots,N-1,\\
\hat\chi_{\mB\mA}&:=\chi_{\mB\mA},\qquad \mB,\mA=2,\dots,N-1,\\
\hat\kappa_{11}&:=\kappa_{11}=\kappa_{1N}=\kappa_{NN},\\
\hat\kappa_{1\mA}&:=\kappa_{1\mA}=\kappa_{N\mA},\qquad \mA=2,\dots,N-1,\\
\hat\kappa_{\mB\mA}&:=\kappa_{\mB\mA},\qquad \mB,\mA=2,\dots,N-1,\\
\hat m_{1\mA}&:=m_{1\mA}=m_{N\mA},\qquad \mA=2,\dots,N-1,\\
\hat m_{\mB\mA}&:=m_{\mB\mA},\qquad \mB,\mA=2,\dots,N-1.
\end{align}
\end{subequations}
\end{theorem}

The second property concerns \emph{reduction on simplex faces}. If one phase is identically zero, then the governing equations should coincide with those of the reduced system obtained by deleting that phase. In other words, the closure remains consistent when the dynamics is restricted to a face of the Gibbs simplex.

\begin{theorem}[Reduction on simplex faces {\cite{mixtureaware2026}}]\label{thm:hier_face}
If a solution of the $N$-phase NSCH system satisfies $\phi_\mA\equiv 0$ for some $\mA$, then the restriction of the $N$-phase system to the Gibbs-simplex face
\begin{align}
\mathcal{G}_\mA:=\{\boldsymbol{\phi}\in\mathcal{G}\;|\;\phi_\mA=0\}
\end{align}
coincides with the $(N-1)$-phase NSCH system obtained by deleting phase $\mA$ and canonically relabeling the remaining indices.
\end{theorem}

The third property is the dynamical consequence of the second one: once a phase is absent, it remains absent throughout the evolution.

\begin{corollary}[Absence invariance {\cite{mixtureaware2026}}]\label[corollary]{cor:hier_absence}
If a phase is absent initially, i.e.\ $\phi_\mA(\cdot,0)\equiv 0$, then it remains absent for all times,
\begin{align}
\phi_\mA(\cdot,t)\equiv 0 \qquad \text{for all }t\ge 0.
\end{align}
\end{corollary}

%These three properties are PDE-level reduction statements satisfied by the continuum mixture-aware closure, and they provide natural structural targets for discretization. In the remainder of this paper, we construct a structure-preserving method that retains the conservation and dissipation structure of the mixture-aware NSCH system while remaining compatible with these reduction mechanisms.

\section{Equivalent alternative formulation}\label{sec:reform}

In this section we provide an alternative but equivalent formulation of the system \eqref{eq:sys2}. This forms the basis for the structure-preserving discretization provided in \cref{sec:scheme}. In \cref{subsec:strong_form} we provide the strong formulation, and subsequently in \cref{subsec:var_form} the corresponding variational formulation.

\subsection{Strong formulation}\label{subsec:strong_form}
The energy evolution \eqref{eq: global energy evolution} of the system \eqref{eq:sys2} results from multiplying the balance laws with particular weights.
\begin{lemma}[Energy evolution]\label[lemma]{lem: alternative energy evolution}
The energy evolution of system~\eqref{eq:sys2} is obtained by taking a linear combination of equations~\eqref{eq:sys2: mom}--\eqref{eq:sys2: chem}, weighted respectively by $\vv$, $-\frac{1}{2}|\vv|^2 \rho_\mA + \mu_\mA + g y \rho_\mA$, $\lambda$, and $-\partial_t \phi_\mA$.
\end{lemma}
\begin{proof}
  See Appendix \ref{appendix: energy}.
\end{proof}

\cref{lem: alternative energy evolution} demonstrates that the standard variational formulation associated with \eqref{eq:sys2} is energy-stable, provided the required test functions belong to the appropriate function spaces. However, the weights corresponding to the mass balance equations -- specifically, $- \frac{1}{2} \rho_\mA |\vv|^2 + \mu_\mA + g y \rho_\mA$ -- do not lie in standard function spaces due to the presence of the nonlinear kinetic term $- \frac{1}{2}\rho_\mA |\vv|^2$. To address this issue, we introduce an equivalent formulation in which the energy evolution is derived using test functions drawn from conventional Sobolev spaces, thereby ensuring compatibility with standard finite element discretizations.

The nonlinear kinetic weight term $- \frac{1}{2}\rho_\mA |\vv|^2$ accounts for the contribution of the mass balance law to the energy evolution. Hence, the nonlinear contribution of the mass balance weight vanishes when substituting the mass balance equations~\eqref{eq:sys2: mass} into the momentum balance equation~\eqref{eq:sys2: mom}. This provides an alternative form of the momentum equation:
\begin{align}\label{eq:ID: mom}
    0=&~\dt(\rho\vv) + \div(\rho\vv\otimes\vv) - \div \mathbf{S} + \nabla \lambda + \displaystyle\sum_\mB\phi_\mB\nabla\mu_\mB  - \rho\mathbf{g}\nn\\
    &~- \frac{1}{2}\vv \displaystyle\sum_\mA \rho_\mA \left(\dt\phi_\mA + \div(\phi_\mA\vv) + \rho_\mA^{-1}\div \bJ_\mA\right)\nn\\
    =&~ \frac{\vv}{2}\dt\rho + \rho\dt\vv + \frac{1}{2} \vv \div(\rho \vv) + \rho \vv\cdot \nabla \vv - \div \mathbf{S} + \nabla \lambda + \displaystyle\sum_\mB\phi_\mB\nabla\mu_\mB - \rho\mathbf{g},
\end{align}
where we have used the identities:
\begin{subequations}
    \begin{align}
        \dt(\rho\vv) - \frac{1}{2}\vv\sum_\mA \dt\trho_\mA =&~\frac{\vv}{2}\dt\rho + \rho\dt\vv, \\
        \div(\rho\vv\otimes\vv) - \frac{1}{2}\vv \sum_\mA\div(\trho_\mA\vv) =&~\frac{1}{2} \vv \div(\rho \vv) + \rho \vv\cdot \nabla \vv,\\
        \sum_\mA\div \bJ_\mA =&~ 0.
    \end{align}
\end{subequations}
Thus, the system comprised of the balance laws \eqref{eq:sys2: mass}-\eqref{eq:sys2: chem} and the latter expression of \eqref{eq:ID: mom} provides an energy-stable formulation where the weights -- specifically, $\vv$, $\mu_\mA + g y \rho_\mA$, $\lambda$, and $-\partial_t \phi_\mA$ -- are members of standard Sobolev spaces.

In what follows, we introduce another modification that is essential for the development of the fully-discrete energy-stable system. By substituting the identity $g_\mA = \rho_\mA^{-1}(\mu_\mA + \lambda)$ along with relation \eqref{eq:ID: mom}, we reformulate system \eqref{eq:sys1} into an equivalent strong form:
\begin{subequations}\label{eq:sys3}
  \begin{align}
 \rho\dt\vv + \frac{1}{2} \vv \div(\rho \vv) + \rho \vv\cdot \nabla \vv + \sum_\mB \trho_\mB \nabla g_\mB  + (1-\sum_\mB \phi_\mB)\nabla \lambda 
    %&\nn\\
    - {\rm div} \boldsymbol{\tau}-\rho\mathbf{g} &=~ 0, \label{eq:sys3: mom}\\
  \partial_t \phi_\mA  + {\rm div}(\phi_\mA  \bv) +\rho_\mA^{-1}{\rm div} (\bJ_\mA )&=~0,\label{eq:sys3: mass}\\
  {\rm div}\bv + \displaystyle\sum_{\mA}  \rho_\mA^{-1} \nabla \cdot \bJ_\mA &=~ 0, \label{eq:sys3: div}\\
  \rho_\mA g_\mA - \dfrac{ \partial \Psi}{\partial \phi_\mA} + {\rm div}\dfrac{\partial \Psi}{\partial\nabla \phi_\mA} - \lambda&=~0.\label{eq:sys3: chem}
  \end{align}
\end{subequations}
System \eqref{eq:sys3} contains $2N+2$ equations for $2N+2$ (independent) state variables ($\bv$, $\left\{\phi_\mA\right\}_{\mA=1,...,N}$, $\lambda$, $\left\{g_\mA\right\}_{\mA=1,...,N}$).

\begin{lemma}[Energy evolution alternative formulation]
The energy evolution of the system \eqref{eq:sys3} follows from a linear combination of \eqref{eq:sys3: mom}-\eqref{eq:sys3: chem} with the weights: $\vv$, $g_\mA\rho_\mA + g y \rho_\mA-\lambda$, $\lambda$, and $-\partial_t \phi_\mA$.
\end{lemma}
\begin{proof}
  Taking the inner product of the momentum equation \eqref{eq:sys3: mom} with $\vv$ yields:
  \begin{align}\label{eq:proof1: mom}
    0 &= \partial_t K + \div (K \vv) - \vv \cdot \div \mathbf{S}  + \displaystyle\sum_\mA\trho_\mA \vv\cdot \nabla g_\mA + (1-\sum_\mB \phi_\mB)\vv \cdot \nabla \lambda- \rho\vv\cdot\mathbf{g},
  \end{align}
  where we have used the identity:
  \begin{subequations}
    \begin{align}
      \partial_t K +  \div (K \vv) =&~ \vv\cdot \left( \frac{\vv}{2}\dt\rho + \rho\dt\vv +\frac{1}{2} \vv \div(\rho \vv) + \rho \vv\cdot \nabla \vv \right).
    \end{align}
  \end{subequations}
  Next, multiplying \eqref{eq:sys3: mass} with $\rho_\mA g_\mA+\rho_\mA g y-\lambda$ and \eqref{eq:sys3: chem} with $-\partial_t \phi_\mA$, and subsequently summing over $\mA$ yields:
  \begin{align}\label{eq:proof1: phase+chem}
      0 = &~\partial_t (G + \Psi) + \sum_\mA (g_\mA\rho_\mA+gy \rho_\mA-\lambda) \div(\phi_\mA\vv) \nn\\
      &~+ \sum_\mA (g_\mA- \rho_\mA^{-1}\lambda)\div \bJ_\mA - \div \left(\sum_\mA \partial_t \phi_\mA \dfrac{\partial \Psi}{\partial \nabla \phi_\mA} \right),
  \end{align}
  where we have used the identities:
  \begin{subequations}
      \begin{align}
      \partial_t G = &~ gy \sum_\mA \rho_\mA \partial_t \phi_\mA, \\
      \partial_t \Psi - \div \left( \partial_t \phi_\mA \dfrac{\partial \Psi}{\partial \nabla \phi_\mA} \right) =&~ \sum_\mA\partial_t \phi_\mA \dfrac{\partial \Psi}{\partial \phi_\mA} - \sum_\mA\partial_t \phi_\mA \div \left(\dfrac{\partial \Psi}{\partial \nabla \phi_\mA}\right) ,\\
      \sum_\mA gy  \div \bJ_\mA  =&~0.
  \end{align}
  \end{subequations}
  Multiplying \eqref{eq:sys2: div} with $\lambda$ provides:
  \begin{align}\label{eq:proof1: div}
 0 = \lambda\div \vv  + \lambda\displaystyle\sum_\mA \rho_\mA^{-1} \div \bJ_\mA.
 \end{align}
Addition of \eqref{eq:proof1: mom}, \eqref{eq:proof1: div},  and \eqref{eq:proof1: phase+chem} gives:
\begin{align}
  \partial_t (K + G + \Psi)  + \div ((K+\Psi + G) \vv) - \div (\mathbf{S}\vv)  +  \div ( \vv^T \mathbf{K} ) & \nn\\
  + \div( (1-\sum_\mB \phi_\mB)\vv \lambda )  + \div (\sum_\mA g_\mA \bJ_\mA)- \div \left(\sum_\mA \dot{\phi}_\mA \dfrac{\partial \Psi}{\partial \nabla \phi_\mA} \right) &~= -\mathbf{S}:\nabla \vv + \sum_\mA \nabla g_\mA \cdot \bJ_\mA,
\end{align}
  where we have used the identities:
  \begin{subequations}
      \begin{align}
     \div (G \vv) = &~- \rho\vv\cdot\mathbf{g} +  \sum_\mA (gy \rho_\mA) \div(\phi_\mA\vv),\\
      \sum_\mA g_\mA \div \bJ_\mA = &~ - \sum_\mA \nabla g_\mA \cdot \bJ_\mA + \sum_\mA \div (g_\mA \bJ_\mA),\\
      \sum_\mA \div (g_\mA \trho_\mA \vv) = &~ \sum_\mA \trho_\mA \vv\cdot \nabla g_\mA + g_\mA \rho_\mA \div (\phi_\mA \vv),\\
      \div( (1-\sum_\mB \phi_\mB)\vv \lambda ) =&~  (1-\sum_\mB \phi_\mB)\vv \cdot \nabla \lambda + \lambda\div((1-\sum_\mB\phi_\mB)\vv),\\
      - \div \left(\sum_\mA \partial_t \phi_\mA \dfrac{\partial \Psi}{\partial \nabla \phi_\mA} \right) =&~ - \div \left(\sum_\mA \dot{\phi}_\mA \dfrac{\partial \Psi}{\partial \nabla \phi_\mA} \right) + \div \left( \vv^T \sum_\mA \nabla \phi_\mA \otimes \dfrac{\partial \Psi}{\partial \nabla \phi_\mA} \right),
  \end{align}
  \end{subequations}
and where the (modified) Korteweg tensor $\mathbf{K}$ is given by:
 \begin{align}
     \mathbf{K} = (\displaystyle\sum_\mA (g_\mA \trho_\mA)-\Psi)\mathbf{I} +   \sum_\mA \nabla \phi_\mA \otimes \dfrac{\partial \Psi}{\partial \nabla \phi_\mA},
 \end{align}
 and where $\dot{\phi}_\mA := \partial_t \phi_\mA + \mathbf{v}\cdot \nabla \phi_\mA$ is the material derivative of $\phi_\mA$.
 Integration over $\Omega$ provides:
 \begin{align}
   \dfrac{{\rm d}}{{\rm d}t}\mathcal{E} = -\mathcal{D}(\bv,\left\{g_\mA\right\} ) + \mathcal{B},
 \end{align}
 where $\mathcal{B}$ is the boundary contribution:
 \begin{align}
   \mathcal{B} =&~ - \la (K+G+\Psi) \vv, \mathbf{n} \ra_{\partial\Omega} + \la \mathbf{S} \vv, \mathbf{n} \ra_{\partial\Omega} - \la \vv^T\mathbf{K}, \mathbf{n} \ra_{\partial\Omega} \nn\\
   &~- \la (1-\sum_\mB \phi_\mB)\vv \lambda , \mathbf{n} \ra_{\partial\Omega}  - \la \sum_\mA g_\mA \bJ_\mA , \mathbf{n} \ra_{\partial\Omega} + \la \sum_\mA \dot{\phi}_\mA \dfrac{\partial \Psi}{\partial \nabla \phi_\mA}, \mathbf{n} \ra_{\partial\Omega} 
 \end{align}
 and where we note the identities: 
  \begin{subequations}
      \begin{align}
        \mathbf{S}:\nabla \vv=&~ 2 \nu \left( \nabla^s \bv - \frac{1}{d} ({\rm div} \mathbf{v}) \mathbf{I}\right):\left(\nabla^s \bv - \frac{1}{d} ({\rm div} \mathbf{v}) \mathbf{I}\right)+ \nu \left(\bar{\lambda} + \frac{2}{d}\right)\left({\rm div} \mathbf{v}\right)^2 \geq 0,\\
        -\sum_\mA \nabla g_\mA \cdot \bJ_\mA =&~ \displaystyle\sum_{\mA,\mB} \left(\nabla g_\mA\right)^T M_{\mA\mB} \nabla g_\mB \geq 0.
  \end{align}
  \end{subequations}
\end{proof} 
\subsection{Variational formulation}\label{subsec:var_form}
In the following we only consider the following sets of boundary conditions
\begin{enumerate}
    \item[(A)] $\Omega$ is a hypercube and identified with the $d$-dimensional torus, i.e. we impose periodic boundary conditions.
    \item[(B)] The volume fractions satisfy $\sum_{\mB}\kappa_{\mA\mB}\nabla\phi_\mB\cdot\mathbf{n}\vert_{\partial\Omega}=0$
    \item[(C)] The chemical potentials satisfy $\sum_\mB M_{\mA\mB}\nabla g_\mB \cdot\mathbf{n}\vert_{\partial\Omega}=0$.
    \item[(D)] The velocity satisfies $\vv\vert_{\partial\Omega_1}=\mathbf{0}$ and  $\vv\cdot\mathbf{n}\vert_{\partial\Omega_2}=\mathbf{0}$, where $\partial\Omega=\partial\Omega_1\cup \partial\Omega_2$.
\end{enumerate}
To derive the weak formulation we introduce the notation $\la a,b\ra := \int_\Omega ab$, for arbitrary functions $a,b:\Omega \rightarrow {\R}^k$, $k\in\mathbb{N}_+$. Guided by the alternative strong form of the momentum equation \eqref{eq:ID: mom}, we utilize the skew-symmetric form of the convection term in the weak formulation of the momentum equation:
\begin{align}
   \mathbf{c}_{skw}(\mathbf{u},\vv,\w) := &~ \frac{1}{2}\la (\mathbf{u}\cdot\nabla)\vv,\w \ra  - \frac{1}{2}\la (\mathbf{u}\cdot\nabla)\w,\vv \ra,
\end{align}
where we note the identity:
\begin{align}
    \la \w ,\frac{1}{2} \vv \div(\rho \vv) + \rho \vv\cdot \nabla \vv \ra = \mathbf{c}_{skw}(\rho\vv,\vv,\w).% + \frac{1}{2}\la \rho \vv \otimes \vv , \w \otimes \mathbf{n} \ra_{\partial\Omega}.
\end{align}
With this we can recast the system into a variational formulation.
% \begin{align}
%   \la \dt\phi,\psi \ra &- \la \phi\vv, \nabla\psi\ra + \la \mathbf{M}(\phi)\nabla(\mu+\alpha p),\nabla\psi \ra + \zeta \la m(\phi)(\mu+\alpha p),\psi\ra= 0, \\
%   \la \mu,\xi \ra &- \gamma\la \nabla\phi,\nabla\xi \ra - \la f'(\phi),\xi \ra = 0,\\
%   \la \frac{\vv}{2}\dt\rho &+ \rho\dt\vv,\w \ra +  \mathbf{c}_{skw}(\rho\vv,\vv,\w) + \la \mathbf{S}(\phi,\nabla\vv),\nabla\w \ra \\
%   &- \la p,\div(\w)\ra + \la \phi\nabla\mu,\w \ra  = 0, \\
%   \la \div(\vv),q \ra &+ \alpha\la \mathbf{M}(\phi)\nabla(\mu+\alpha p),\nabla q \ra +\alpha\zeta \la m(\phi)(\mu+\alpha p),q\ra = 0
% \end{align}
Every smooth solution satisfies the variational formulation
\begin{subequations}\label{eq: var form}
\begin{align}
  \la \frac{\vv}{2}\dt\rho + \rho\dt\vv,\w \ra +  \mathbf{c}_{\rm skw}(\rho\vv,\vv,\w) + \la \mathbf{S},\nabla\w \ra - \la \lambda,\div\w \ra&\\
  + \sum_\mA\la \phi_\mA\nabla(\rho_\mA g_\mA-\lambda),\w \ra - \la \rho \mathbf{g}, \w \ra &~=~ 0, \\
  \la \dt\phi_\mA,\psi_\mA \ra - \la \phi_\mA\vv, \nabla\psi_\mA\ra + \frac{1}{\rho_\mA}\sum_\mB\la \mathbf{M}_{\mA,\mB}\nabla g_\mB,\nabla\psi_\mA \ra &~=~ 0, \\
  \la \div(\vv),q \ra + \sum_{\mA,\mB}\frac{1}{\rho_\mA}\la \mathbf{M}_{\mA,\mB}\nabla g_\mB,\nabla q \ra &~=~ 0\\
  \la \rho_\mA g_\mA,\xi_\mA \ra - \la \partial_{\nabla \phi_\mA}\Psi,\nabla\xi_\mA \ra - \la \partial_{\phi_\mA}\Psi,\xi_\mA \ra - \la \lambda,\xi_\mA \ra &~=~ 0,
\end{align}
\end{subequations}
for smooth test functions $(\psi,\xi,\w,q)$ with mean-free $q$.

\begin{lemma}[Structure-preserving properties variational form]\label[lemma]{lem:variational_TC}
 The variational formulation \eqref{eq: var form} satisfies the following structure-preserving properties:
 \begin{subequations}
     \begin{align}
    \text{(phase volume conservation)}  \qquad  &  \ddt \la \phi_\mA(t), 1\ra =~ 0, \label{eq:phase volume conservation}\\
    \text{(phase mass conservation)}  \qquad   &  \ddt \la \trho_\mA(t), 1\ra =~ 0, \label{eq:phase mass conservation}\\
    \text{(energy dissipation)}   \qquad  &  \ddt \mathcal{E}\left(\vv(t),\left\{\phi_\mA(t)\right\}\right)=~ - \mathcal{D}\left(\vv, \left\{g_\mA\right\}\right). \label{eq:energy dissipation}
     \end{align}
 \end{subequations}
\end{lemma}
\begin{proof}
    The conservation of the phase volume results when substituting $\psi_\mA = 1$. Next, phase mass conservation follows when multiplying \eqref{eq:phase volume conservation} by $\rho_\mA$. 
    
    We proceed with the energy dissipation relation. 
    Selecting $\w=\vv$ provides:  
    \begin{align}\label{eq:proof2: mom}
    0 =&~ \dfrac{{\rm d}}{{\rm d}t} \la K ,1\ra + \la \mathbf{S},\nabla\vv \ra - \la \lambda,\div \vv \ra + \sum_\mA \la \phi_\mA\nabla (\rho_\mA g_\mA - \lambda),\vv \ra + \la \rho g \mathbf{j}, \vv\ra,
  \end{align}
  where we have used the identities:
  \begin{subequations}
    \begin{align}
  \dfrac{{\rm d}}{{\rm d}t} \la K ,1\ra   =&~ \la \tfrac{1}{2}\vv\partial_t\rho + \rho\partial_t\vv,\vv \ra, \\
  0 = &~  \mathbf{c}_{\rm skw}(\rho\vv,\vv,\vv).
    \end{align}
  \end{subequations}
  Next, taking $\psi_\mA = \rho_\mA g_\mA+ \rho_\mA g y-\lambda$ and $\xi_\mA = -\partial_t \phi_\mA$, and subsequently summing over $\mA$ yields:
\begin{align}\label{eq:proof2: phase}
 \dfrac{{\rm d}}{{\rm d}t} \la G + \Psi ,1\ra  - \sum_\mA \la \phi_\mA\vv, \nabla(\rho_\mA g_\mA)\ra - \sum_\mA \la \phi_\mA\vv, \nabla(\rho_\mA g y)\ra+ \sum_\mA \la \phi_\mA\vv, \nabla\lambda\ra&\nn\\
  + \sum_{\mA,\mB}\la \mathbf{M}_{\mA,\mB}\nabla g_\mB,\nabla g_\mA \ra    - \sum_{\mA,\mB}\la \mathbf{M}_{\mA,\mB}\nabla g_\mB,\nabla (\rho_\mA^{-1}\lambda) \ra &~=~ 0,    
\end{align}
where we have used the identities:
  \begin{subequations}
      \begin{align}
      \dfrac{{\rm d}}{{\rm d}t} \la G ,1\ra = &~ \sum_\mA \la \dt\phi_\mA,\rho_\mA g y \ra \\
      \dfrac{{\rm d}}{{\rm d}t} \la \Psi ,1\ra =&~ \sum_\mA\la \partial_{\nabla \phi_\mA}\Psi,\nabla \partial_t \phi_\mA \ra +\sum_\mA \la \partial_{\phi_\mA}\Psi,\partial_t \phi_\mA \ra,\\
      \sum_{\mA,\mB} \la \mathbf{M}_{\mA,\mB}\nabla g_\mB, g \mathbf{j} \ra  =&~0.
  \end{align}
  \end{subequations}
  Next, taking $q=\lambda$ provides:
  \begin{align}\label{eq:proof2: div}
 0=&~\la \div\vv, \lambda \ra + \sum_{\mA,\mB}\frac{1}{\rho_\mA}\la \mathbf{M}_{\mA,\mB}\nabla g_\mB,\nabla \lambda \ra ,
 \end{align}
 where we used $\la \div(\vv),\bar y \ra =0 $ and $\nabla \bar y =0$, and where $\bar{y}:=\la y,1\ra$ denotes the mean-value of $y$.
 Addition of \eqref{eq:proof2: mom} and \eqref{eq:proof2: phase} and \eqref{eq:proof2: div}, provides:
  \begin{align}
      \dfrac{{\rm d}}{{\rm d}t} \mathcal{E} =&~  - \la \mathbf{S},\nabla\vv \ra  - \sum_{\mA,\mB}\la \mathbf{M}_{\mA,\mB}\nabla g_\mB,\nabla g_\mA \ra = - \mathcal{D}\left(\vv, \left\{g_\mA\right\}\right),
  \end{align}
  where we have used the identity:
\begin{align}
   - g \sum_\mA \rho_\mA  \la \phi_\mA\vv, \nabla y\ra  + \la \rho g \mathbf{j}, \vv\ra % = - g \sum_\mA \rho_\mA  \la \phi_\mA\vv, \nabla y\ra + \sum_\mA \la \rho_\mA \phi_\mA g \nabla y, \vv\ra
   = 0.
\end{align}
\end{proof}

The variational formulation inherits the mixture-aware properties provided in \cref{sec:hierarchy_consistency}.

\begin{lemma}[Mixture-aware properties of the variational formulation]
\label[lemma]{lem:mixture_aware_variational}
Every smooth solution of the variational formulation \eqref{eq: var form} inherits the three mixture-aware properties of the continuum model provided in \cref{sec:hierarchy_consistency}.
\end{lemma}

\begin{proof}
Because the solutions of the variational formulation are assumed smooth, the identities established pointwise in the proofs of \cref{thm:hier_merge,thm:hier_face}, and \cref{cor:hier_absence}
may be inserted directly into the duality pairings of \eqref{eq: var form}. In this sense, the present proof is the variational counterpart of the corresponding strong-form reductions.
\end{proof}

\section{Numerical scheme \& structural properties}\label{sec:scheme}
In this section, we will propose a fully-discrete finite element method and prove our main result on the structure-preservation properties. As a preparatory step, we first introduce the time-discretization in \cref{subsec:time}, and subsequently the spatial discretization in \cref{subsec:space}. 

\subsection{Time discretization}\label{subsec:time}
We partition the time interval $[0,T]$ with a family of time parameter $\{\tau_k\}>0$ and introduce $\Itau:=\{0=t^0,t^1=\tau_1,\ldots, t^{n_T}=T=\sum_k \tau_k\}$, where $n_T$ is the absolute number of time steps. We denote by $\Pi^1_c(\Itau), \Pi^0(\Itau)$ the spaces of continuous piece-wise linear and piecewise constant functions on $\Itau$. We introduce $g^{n+1}:=g(t^{n+1})$ and $g^n:=g(t^n)$ . We use superscripts to indicate possibly different time-level choices in the discrete terms. The notation $g^\star$ serves as a placeholder for the time level used in the skew-symmetric convection term, $g^\dagger$ for the time level used in the stress tensor, and $g^\ddagger$ for the time level used in the mobility tensor. Typical choices are $g^\star,\ g^\dagger,\ g^\ddagger \in\{g^n,g^{n+1}\}.$ Finally, we introduce the time difference and the discrete time derivative via
\begin{equation}
	 d^{n+1}_\tau g = \frac{g^{n+1}-g^n}{\tau_n}.
\end{equation}

The gradient-free part of the free energy, i.e $\Psi_0(\phi)$, is treated using time-averages in the spirit of \cite{BrunkBiot24}, i.e.
\begin{subequations}
    \begin{align}
 \partial_{\phi_\mA}\Psi(\left\{\phi_\mA^n\right\},\left\{\phi_\mA^{n+1}\right\}) :=&~ \frac{1}{\tau_n}\int_{t^{n}}^{t^{n+1}} \partial_{\phi_\mA}\Psi(\left\{\phi_\mA(s)\right\}) ds, \\
 \left\{\phi_\mA(s)\right\} =&~ (\left\{\phi_\mA^{n+1}\right\}-\left\{\phi_\mA^n\right\})\frac{(s-t^{n+1})}{\tau_n} + \left\{\phi_\mA^{n+1}\right\}.  \label{eq:timeavg} 
\end{align}
\end{subequations}
We note that one can also employ the well-known convex-concave splitting. 

Next, we introduce three regularizations. First, accommodating that $\phi_{\mA}$ might leave the desired interval $[0,1]$, we propose a regularization of the density $\rho$ as follows:
\begin{equation}
 \tilde\rho:=\sum_{\alpha} \rho_\mA\tilde\phi_\mA ,\qquad  \tilde \phi_{\mA} := \begin{cases}
     \phi_{\mA} , \phi_{\mA}\in[\delta,1] \\
     1, \phi_{\mA} \geq 1, \\
     \delta, \phi_{\mA} \leq \delta.
 \end{cases} 
\end{equation}
Second, we also use the regularization of the volume fractions in the mobility matrix:
\begin{align}
\tilde{M}_{\alpha\beta}(\boldsymbol{\phi})=&~-m_{\alpha\beta}\tilde{\phi}_\alpha\tilde{\phi}_\beta \quad (\alpha\neq\beta),
\end{align}
which is compatible with the row and column summation constraints of the mobility matrix.
Lastly, we utilize the polynomial regularization of the mixture-aware bulk free energy introduced in \cite{mixtureaware2026}.

We now propose the structure-preserving time-integration scheme:
\begin{subequations}\label{eq:sys time}
  \begin{align}
 \frac{\vv^{n+1}}{2}d^{n+1}_\tau \tilde{\rho} + \tilde{\rho}^n d^{n+1}_\tau\vv + \frac{1}{2} \vv^{n+1} \div(\rho^* \vv^*) + \rho^* \vv^*\cdot \nabla \vv^{n+1}&\nn\\
 - \div \mathbf{S}(\left\{\phi_\mA^{\dagger}\right\},\nabla \vv^{n+1}) + \nabla \lambda^{n+1}  + \sum_\mA \phi_\mA^n\nabla(\rho_\mA g_\mA^{n+1}-\lambda^{n+1}) -\rho(\left\{\phi_\mA^n\right\}) \mathbf{g} &~=0,\label{eq:sys time: mom}\\
 d^{n+1}_\tau\phi_\mA + \div(\phi_\mA^n\vv^{n+1}) - \div\left(\frac{1}{\rho_\mA}\sum_\mB\tilde{M}_{\mA,\mB}^{\ddagger} \nabla g_\mB^{n+1}\right) &~= 0, \label{eq:sys time: phase}\\
 \rho_\mA g_\mA^{n+1} - \partial_{\phi_\mA}\Psi(\left\{\phi_\mA^n\right\},\left\{\phi_\mA^{n+1}\right\}) + {\rm div} (\partial_{\nabla \phi_\mA} \Psi(\left\{\phi_\mA^{n+1}\right\})) - \lambda^{n+1} &~=0,\label{eq:sys time: chem}\\
 \div (\vv^{n+1})  -  \sum_{\mA,\mB} \div \left(\frac{1}{\rho_\mA}\tilde{M}_{\mA,\mB}^{\ddagger} \nabla g_\mB^{n+1}\right) &~= 0, \label{eq:sys time: div}
\end{align}
\end{subequations}

Note that one can show that solutions of \eqref{eq:sys time} conserve (total) phase volume, (total) phase mass, and satisfy a discrete energy dissipation law with respect to the regularized total energy
\begin{subequations}
  \begin{align}
    \tilde{\mathcal E}\bigl(\vv,\{\phi_\mA\}\bigr):=&~
\int_\Omega \Psi\bigl(\{\phi_\mA\},\{\nabla \phi_\mA\}\bigr)
+\tilde{K}(\left\{\phi_\mA\right\},\vv) + G(\left\{\phi_\mA\right\}),\\
\tilde{K}(\left\{\phi_\mA\right\},\vv) :=&~\frac{\tilde\rho}{2}|\vv|^2.
\end{align}
\end{subequations}
The proof follows the same lines as \cref{lem:variational_TC} and is postponed to the proof for the fully-discrete scheme, i.e. \cref{thm:discreten_stable} which we will consider below.

\subsection{Space discretization}\label{subsec:space}
For the spatial discretisation we require that $\Th$ is a geometrically conforming partition of $\Omega$ into elements where $h$ is the maximal diameter of the elements in $\Th$.
%
% This assumption could again be relaxed to allow for adaptive mesh refinement.
%
The space of continuous and piecewise linear and quadratic functions over $\Th$ as well as the mean free are introduced via
\begin{subequations}\label{eq:defFespace}
\begin{align}
	\Vh &:= \{v \in H^1(\Omega)\cap C^0(\bar\Omega) : v|_K \in P_1(K) \quad \forall K \in \Th\},\\
	\Xh &:= \{\vv \in H^1(\Omega)^d\cap C^0(\bar\Omega)^d : \vv|_K \in P_2^d(K) \quad \forall K \in \Th\}\cap\{\vv\vert_{\partial\Omega_1}=0, \vv\cdot \mathbf{n}\vert_{\partial\Omega_2}=0  \},\\
	\Qh &:= \{v \in \Vh : \la v, 1\ra=0\}.
\end{align}
\end{subequations}
Here $P_k(K)$ denote the space of polynomials with maximal degree $k$ on $K$.
In the case of periodic boundary conditions, the boundary conditions incorporated in $\Xh$ are omitted.
 
This variational formulation allows to directly deduce a structure-preserving approximation.

\begin{problem}\label{prob:scheme}
    Let $(\{\phi_{\mA,h}^0\},\vv_{h}^0)\in \Vh^N\times\Xh$ be given. We seek functions $(\{\phi_{\mA,h}\},\vv_h)\in P_1^c(\Itau;\Vh^N\times\Xh)$ and $(\{g_{\mA,h}\},\lambda_h)\in P^0(\Itau;\Vh^N\times\Qh)$ such that
\begin{subequations}
    \begin{align}
    \la \tfrac{\vv_h^{n+1}}{2}\dtau\tilde\rho_h + \tilde\rho_h^n\dtau\vv_h,\w_h \ra +  \mathbf{c}_{\rm skw}(\rho_h^*\vv_h^*,\vv_h^{n+1},\w_h)+ \la \mathbf{S}(\{\phi_{\mA,h}^{\dagger}\},\nabla\vv_h^{n+1}),\nabla\w_h \ra & \nn \\
- \la \lambda_h^{n+1},\div \w_h \ra + \sum_\mA \la \phi_{\mA,h}^n \nabla(\rho_\mA g_{\mA,h}^{n+1}-\lambda_h^{n+1}),\w_h \ra + \la \rho(\left\{\phi_{\mA,h}^n\right\})g\mathbf{j}, \w_h\ra &~= 0, \label{eq:scheme3}\\
\la \dtau\phi_{\mA,h},\psi_{\mA,h} \ra - \la \phi_{\mA,h}^n\vv_h^{n+1}, \nabla\psi_{\mA,h}\ra + \la \frac{1}{\rho_\mA}\sum_\mB \tilde{M}_{\mA,\mB}^{\ddagger} \nabla g_{\mB,h}^{n+1},\nabla\psi_{\mA,h} \ra &~= 0, \label{eq:scheme1}\\
  \la \rho_\mA g_{\mA,h}^{n+1},\xi_{\mA,h} \ra- \la \partial_{\phi_{\mA,h}}\Psi(\{\phi_{\mA,h}^n\},\{\phi_{\mA,h}^{n+1}\}),\xi_{\mA,h} \ra - \la \partial_{\nabla \phi_{\mA,h}} \Psi(\{\phi_{\mA,h}^{n+1}\}),\nabla\xi_{\mA,h} \ra & \nn\\
  - \la \lambda_h^{n+1},\xi_{\mA,h} \ra &~= 0,\label{eq:scheme2}\\
  \la \div\vv_h^{n+1},q_h \ra + \la \sum_{\mA,\mB}  \frac{1}{\rho_\mA}\tilde{M}_{\mA,\mB}^{\ddagger} \nabla g_{\mB,h}^{n+1},\nabla q_h \ra &~= 0, \label{eq:scheme4}
\end{align}
\end{subequations}
holds for all $(\{\psi_{\mA,h}\},\{\xi_{\mA,h}\},\w_h,q_h)\in\Vh^N\times\Vh^N\times\Xh\times\Qh$ and for all $0\leq n < n_T.$
\end{problem}

\begin{theorem}[Fully-discrete structure-preserving properties]\label[theorem]{thm:discreten_stable}
    Every solution of Problem \ref{prob:scheme} satisfies the following structure-preserving properties:
 \begin{subequations}
     \begin{align}
    \text{(phase volume conservation)}  \qquad  &  \la \phi_{\mA,h}^{n+1},1 \ra = \la \phi_{\mA,h}^0,1 \ra, \label{eq:phase volume conservation fully-discrete}\\
    \text{(phase mass conservation)}  \qquad  &  \la \trho_{\mA,h}^{n+1},1 \ra = \la \trho_{\mA,h}^0,1 \ra, \label{eq:phase mass conservation fully-discrete}\\
    \text{(total volume conservation)}  \qquad  &  \la \sum_\mA \phi_{\mA,h}^{n+1},1 \ra = \la \sum_\mA\phi_{\mA,h}^0,1 \ra, \label{eq:total volume conservation fully-discrete}\\
    \text{(total mass conservation)}  \qquad  &  \la \rho_h^{n+1},1 \ra = \la \rho_{h}^0,1 \ra, \label{eq:total mass conservation fully-discrete}\\
    \text{(energy dissipation)}   \qquad  &   \tilde{\mathcal{E}}\left(\vv_h^{n+1},\left\{\phi_{\mA,h}^{n+1}\right\}\right) + \tau_n \mathcal{D}_h^{n+1}\leq\tilde{\mathcal{E}}\left(\vv_h^{n},\left\{\phi_{\mA,h}^{n}\right\}\right), \label{eq:energy dissipation  fully-discrete}
     \end{align}
 \end{subequations}   
for all $\mA = 1,...,N$, and for all $0 \leq n\leq n_T-1$, where the physical dissipation is:
\begin{align}
  \mathcal{D}_h^{n+1}:= \la \mathbf{S}(\{\phi_\mA^{\dagger}\},\nabla\vv_h^{n+1}),\nabla\vv_h^{n+1} \ra + \sum_{\mA,\mB}\la \tilde{M}_{\mA,\mB}^{\ddagger}\nabla g_{\mB,h}^{n+1},\nabla g_{\mA,h}^{n+1} \ra.
\end{align}
In addition, if the volume saturation constraint holds for the initial data, then it holds for all time steps, i.e. 
    \begin{align}\label{eq: saturation constraint discrete}
 \sum_\mA \phi_{\mA,h}^0(x) = 1 \qquad \Longrightarrow \qquad \sum_\mA \phi_{\mA,h}^k(x) = 1, \qquad k=1,...,n_T, x \in \Omega.
    \end{align}
\end{theorem}
\begin{proof}
    Conservation of phase volume follows again by insertion of $\psi_{\mA,h}=1\in\Vh$, and conservation of phase mass follows from multiplying \eqref{eq:phase volume conservation fully-discrete} with $\rho_\mA$. Next, conservation of total volume and total mass results from summing \eqref{eq:phase volume conservation fully-discrete} and \eqref{eq:phase mass conservation fully-discrete} over $\mA$.
    
    We proceed with the energy dissipation property. Taking $\w_h=\vv_h^{n+1}\in\Xh$ provides:  
    \begin{align}\label{eq:proof3: mom}
    0 =&~ \frac{1}{\tau_n}\la \tilde{K}(\left\{\phi_{\mA,h}^{n+1}\right\},\vv_h^{n+1}) - \tilde{K}(\left\{\phi_{\mA,h}^{n}\right\},\vv_h^{n}),1\ra + \tau_n\frac{\tilde\rho_h^n}{2}\norm{\dtau\vv_h}_0^2 + \la \mathbf{S}(\{\phi_\mA^{\dagger}\},\nabla\vv_h^{n+1}),\nabla\vv_h^{n+1} \ra \nn\\
    &~- \la \lambda_h^{n+1},\div \vv_h^{n+1} \ra + \sum_\mA \la \phi_{\mA,h}^n \nabla(\rho_\mA g_{\mA,h}^{n+1}-\lambda_h^{n+1}),\vv_h^{n+1} \ra + \la \rho(\left\{\phi_{\mA,h}^n\right\})g\mathbf{j}, \vv^{n+1}_h\ra,
  \end{align}
  where we have used the identities:
  \begin{subequations}
    \begin{align}
    \la \tfrac{1}{2}\vv_h^{n+1}\dtau\tilde\rho_h + \tilde\rho_h^n\dtau\vv_h^{n+1},\vv_h^{n+1} \ra  =&~\frac{1}{\tau_n}\la \tilde{K}(\left\{\phi_{\mA,h}^{n+1}\right\},\vv_h^{n+1}) - \tilde{K}(\left\{\phi_{\mA,h}^{n}\right\},\vv_h^{n}),1\ra \nn\\
    &~+ \tau_n\frac{\tilde\rho_h^n}{2}\norm{\dtau\vv_h}_0^2\\
    \mathbf{c}_{\rm skw}(\rho_h^*\vv_h^*,\vv_h^{n+1},\vv_h^{n+1}) = &~0.
    \end{align}
  \end{subequations}
   Next, we select $\psi_{\mA,h} = \rho_\mA g_{\mA,h}^{n+1}+\rho_\mA gy - \lambda^{n+1}_h \in \Vh$ and $\xi_{\mA,h}=-\dtau \phi_{\mA,h} \in \Vh$, summing over $\mA$, and subsequently adding the results provides:

    \begin{align}\label{eq:proof3: phase}
     0=&~  \sum_\mA\la \partial_{\phi_{\mA,h}}\Psi_0(\left\{\phi_{\mA,h}^n\right\},\{\phi_{\mA,h}^{n+1}\}),\dtau \phi_{\mA,h} \ra + \sum_\mA \la \partial_{\nabla \phi_{\mA,h}} \Psi_\nabla(\{\phi_{\mA,h}^{n+1}\}),\nabla\dtau \phi_{\mA,h} \ra \nn\\
     &~- \sum_\mA \la \phi_{\mA,h}^n\vv_h^{n+1}, \nabla(\rho_\mA g_{\mA,h}^{n+1})\ra+ \sum_\mA \la \phi_{\mA,h}^n\vv_h^{n+1}, \nabla\lambda_h^{n+1}\ra \nn\\
     &~+ \sum_{\mA,\mB}\la \tilde{M}_{\mA,\mB}^{\ddagger}\nabla g_{\mB,h}^{n+1},\nabla g_{\mA,h}^{n+1} \ra    - \sum_{\mA,\mB}\la \tilde{M}_{\mA,\mB}^{\ddagger}\nabla g_{\mB,h}^{n+1},\nabla (\rho_\mA^{-1}\lambda_h^{n+1}) \ra\nn\\
      &~+\frac{1}{\tau_n}\la G(\{\phi_{\mA,h}^{n+1}\}) - G(\{\phi_{\mA,h}^n\}),1\ra  - \sum_\mA \la \phi_{\mA,h}^n\vv_h^{n+1}, \nabla(\rho_\mA g y)\ra,
    \end{align}
  where we have used the identities:
  \begin{subequations}
      \begin{align}
       \sum_\mA \la \dtau\phi_{\mA,h},\rho_\mA gy  \ra = &~ \frac{1}{\tau_n}\la G(\{\phi_{\mA,h}^{n+1}\}) - G(\{\phi_{\mA,h}^n\}),1\ra,\\
      \sum_{\mA,\mB} \la \tilde{M}_{\mA,\mB}^{\ddagger}\nabla g_{\mB,h}^{n+1}, g \mathbf{j} \ra  =&~0.       
  \end{align}
  \end{subequations}
  Lastly, taking $q_h=\lambda_h^{n+1} \in \Qh$ provides:
  \begin{align}\label{eq:proof3: div}
 0=&~\la \div\vv_h^{n+1}, \lambda_h^{n+1} \ra + \sum_{\mA,\mB}\frac{1}{\rho_\mA}\la \tilde{M}_{\mA,\mB}^{\ddagger}\nabla g_{\mB,h}^{n+1},\nabla \lambda_h^{n+1} \ra, 
 \end{align}
 where we used that $\la \div(\vv_h^{n+1}),\bar y \ra=0$ and $\nabla\bar y =0$. Addition of \eqref{eq:proof3: mom}, \eqref{eq:proof3: phase}, and \eqref{eq:proof3: div} provides:
  \begin{align}
      \frac{1}{\tau_n}(\tilde{\mathcal{E}}(\{\phi_{\mA,h}^{n+1}\},\vv_h^{n+1}) - \tilde{\mathcal{E}}\left(\left\{\phi_{\mA,h}^{n}\right\},\vv_h^{n})\right) =&~  - \la \mathbf{S}(\{\phi_\mA^{\dagger}\},\nabla\vv_h^{n+1}),\nabla\vv_h^{n+1} \ra - \sum_{\mA,\mB}\la \tilde{M}_{\mA,\mB}^{\ddagger}\nabla g_{\mB,h}^{n+1},\nabla g_{\mA,h}^{n+1} \ra \nn\\
      &~  - \tau_n\frac{\tilde\rho_h^n}{2}\norm{\dtau\vv_h}_0^2-\frac{\tau_n}{2}\sum_{\mA\mB} \la \kappa_{\alpha\beta}\nabla\dtau \phi_{\mA,h},\nabla\dtau \phi_{\mB,h} \ra \nn\\
 \leq &~ - \mathcal{D}_h^{n+1},
  \end{align}
  where the inequality follows from the positive semi-definiteness of $\kappa_{\mA\mB}$ and the positivity of $\tilde\rho$. Furthermore we used the relations:
  \begin{subequations}
\begin{align}
    - \sum_\mA \la \phi_{\mA,h}^n\vv_h^{n+1}, \nabla(\rho_\mA g y)\ra   + \la \rho(\left\{\phi_{\mA,h}^n\right\})g\mathbf{j}, \vv^{n+1}_h\ra =&~ 0,\\
    %&\sum_\mA\la \partial_{\phi_{\mA,h}}\Psi(\{\phi_{\mA,h}^n\},\{\phi_{\mA,h}^{n+1}\}),\dtau \phi_{\mA,h} \ra + \la \partial_{\nabla \phi_{\mA,h}} \Psi(\{\phi_{\mA,h}^{n+1}\}),\nabla\dtau \phi_{\mA,h} \ra \nn \\
    %&= \frac{1}{\tau_n}\la \Psi(\{\phi_{\mA,h}^{n+1}\}) - \Psi(\left\{\phi_{\mA,h}^n\right\}),1\ra  -\frac{\tau_n}{2}\sum_{\mA\mB} \la \kappa_{\alpha\beta}\nabla\dtau \phi_{\mA,h},\nabla\dtau \phi_{\mB,h} \ra
    \sum_\mA \la \partial_{\phi_{\mA,h}}\Psi_0(\{\phi_{\mA,h}^n\},\{\phi_{\mA,h}^{n+1}\}),\dtau \phi_{\mA,h} \ra
=&~
\frac{1}{\tau_n}\la \Psi_0(\{\phi_{\mA,h}^{n+1}\})-\Psi_0(\{\phi_{\mA,h}^{n}\}),1\ra,
\\
\sum_\mA
\la \partial_{\nabla \phi_{\mA,h}}\Psi_\nabla(\{\phi_{\mA,h}^{n+1}\}),\nabla \dtau \phi_{\mA,h} \ra
=&~
\frac{1}{\tau_n}\la \Psi_\nabla(\{\phi_{\mA,h}^{n+1}\})-\Psi_\nabla(\{\phi_{\mA,h}^{n}\}),1\ra
\nn\\
&\quad
+\frac{\tau_n}{2}\sum_{\mA,\mB}
\la \kappa_{\mA\mB}\nabla \dtau \phi_{\mA,h},\nabla \dtau \phi_{\mB,h} \ra .
\end{align}\end{subequations}

  Finally, we consider the saturation constraint \eqref{eq: saturation constraint discrete}. Inserting $\psi_{\mA,h} = 1$ and summation over $\alpha$ shows
  \begin{align}
      \la \dtau (\sum_\mA \phi_{\mA,h} -1),1\ra = 0,
  \end{align}
  i.e. the mean value, zero, of the constraint is preserved. Next, we suppose that the constraint holds for time-step $n$ and consider the discrete evolution at time step $n+1$. We insert $\psi_{\mA,h}=\psi_h \in\Qh\subset\Vh$, $q_h=\psi_h\in\Qh$ and sum the phase equation over $\mA$ to find:
\begin{subequations}
    \begin{align}
\la \dtau \sum_\mA\phi_{\mA,h},\psi_{h} \ra - \la \sum_\mA\phi_{\mA,h}^n\vv_h^{n+1}, \nabla\psi_{h}\ra + \la \sum_{\mA,\mB}\frac{1}{\rho_\mA}\tilde{M}_{\mA,\mB}^{\ddagger} \nabla g_{\mB,h}^{n+1},\nabla\psi_{h} \ra &~= 0,\\
 \la \div\vv_h^{n+1},\psi_h \ra + \la \sum_{\mA,\mB}  \frac{1}{\rho_\mA}\tilde{M}_{\mA,\mB}^{\ddagger} \nabla g_{\mB,h}^{n+1},\nabla \psi_h \ra &~= 0.       
    \end{align}
\end{subequations}
Subtracting these equations provides:
\begin{equation}
  \la \dtau \sum_\mA\phi_{\mA,h},\psi_{h} \ra - \la \sum_\mA\phi_{\mA,h}^n\vv_h^{n+1}, \nabla\psi_{h}\ra  =\la \div\vv_h^{n+1},\psi_h \ra.
\end{equation}
Integration by parts and rearranging the equations yields:
\begin{align}
      \la \dtau (\sum_\mA\phi_{\mA,h}-1),\psi_{h} \ra  =-\la(1- \sum_\mA \phi_{\mA,h}^n),  \vv_h^{n+1}\cdot\nabla \psi_{h}\ra ,
\end{align}
where we used the assumed boundary conditions. Since the constraint holds for the old time-step we obtain 
\begin{equation}
   \la (\sum_\mA \phi_{\mA,h}^{n+1} - 1),\psi_h \ra =  0, \qquad \forall \psi_h\in\Qh.
\end{equation}
By mean-freedom we are allowed to take $\psi_h=\sum_{\mA}\phi_{\mA,h}^{n+1} - 1$ and conclude that the constraint holds pointwise everywhere in $\Omega$:
\begin{align}
    \sum_{\mA}\phi_{\mA,h}^{n+1}(x) = 1, \qquad x \in \Omega.
\end{align}
\end{proof}

\begin{remark}[Finite element spaces]
 Note that the results in \cref{thm:discreten_stable} are not restricted to this particular choice of finite element spaces. We only require that $\Vh$ and $\Qh$ are $H^1$-conforming spaces containing at least piecewise-linear functions and that the finite element pair $\Xh\times\Qh$ is inf-sup stable for the (Navier-)Stokes equation.
\end{remark}

%\section{Modeling of the free energy}
\section{Numerical results}\label{sec:numerics}

In this section we test the proposed numerical scheme. First, we consider a convergence test ($N=3$) in \cref{subsec:conv}, and a phase separation test ($N=3$) in \cref{subsec:phasesep}. Next, we consider a well-known rising bubble test case ($N=2$) in \cref{subsec:risingbubble}. Finally, we show a test of a rising bubble through two stratified liquid layers ($N=3$) in \cref{subsec:risingbubblethree}.
%The resulting nonlinear systems are tackled by Newton's method with an tolerance $10^{-\marking{xx}}.$
The resulting nonlinear systems are tackled by Newton's method, and the resulting linear systems are solved using a direct solver. The placeholder quantities are all evaluated at time-step $n+1$. 
The code is implemented in Fenics \cite{Fenics} as well as NGSolve \cite{schoberl2014c++}.

\subsection{Convergence test}\label{subsec:conv}

For the convergence test we consider $N=3$ phases, final time $T=0.1$, and domain $\Omega=[0,1]^2$. For the specific densities we choose $\rho_1 = 1, \rho_2=2, \rho_3=3$, and for the partial viscosities $\eta_1=0.01, \eta_2=0.02, \eta_3=0.03$. For the bulk potential and the mobility we consider the choices from Section \ref{subsec:time} and as capillary matrix we choose 
\begin{equation}
    \kappa = 3\cdot 10^{-5}\begin{pmatrix}
        5.93776285 & -4.93135582 &-1.00640703 \\
 -4.93135582 &  7.79781882 & -2.86646300 \\
 -1.00640703 & -2.86646300 &  3.87287003
    \end{pmatrix}.
\end{equation}
This capillary matrix is computed using the algorithm proposed in \cite{surfacetension2026} and corresponds to the surface tensions $\gamma_{12} = 0.007, \gamma_{13} = 0.005, \gamma_{23}=0.006$. This implies that the corresponding interface widths are
\begin{align}
    \varepsilon_{12}=0.0084, \qquad \varepsilon_{13}=0.0060, \qquad \varepsilon_{23}=0.0071.
\end{align}
The different density values ensure that the test is genuinely quasi-incompressible. Finally we consider periodic boundary conditions and the initial data:
\begin{subequations}
\begin{align}
   \phi_1(x,y)=&~0.3 + 0.21\sin(\pi x)\sin(2\pi y), \\ \phi_2(x,y)=&~ 0.3 + 0.24\cos(4\pi x)\sin(4\pi y)^2, \\
   \phi_3(x,y) =&~ 1-\phi_1(x,y) - \phi_2(x,y), \\ \vv_0(x,y) =&~ 10^{-1}(\sin(\pi x)^2\sin(2\pi y),\sin(\pi y)^2\sin(2\pi x))^\top.
\end{align}
\end{subequations}

%\subsection*{Space Convergence:}
Since no exact solution is available we compute the error using a refined solution for a constant and fixed time step size $\tau=5\cdot 10^{-3}$. The error quantities we consider are
\begin{align*}
   \text{err}(\left\{\phi_{\mA}\right\},h_k)&:= \max_{n\in\Itau}\norm*{\left\{\phi_{\mA,h_k}^n\right\} -\left\{\phi_{\mA,h_{k+1}}^n\right\}}_{L^2(\Omega)}^2 , \qquad \text{err}(\vv,h_k):=\max_{n\in\Itau}\norm{\vv^n_{{h_k}}-\vv^n_{h_{k+1}}}_0^2 \\
   \text{err}(\left\{g_{\mA}\right\},h_k)&:=\tau\sum_{n=1}^{n_T}\norm*{\left\{g_{\mA,h_k}^n\right\} -\left\{g_{\mA,h_{k+1}}^n\right\}}_{H^1(\Omega)}^2 , \qquad \text{err}(\nabla\vv,h_k):=\tau\sum_{n=1}^{n_T}\norm{\vv_{h_k}^n-\vv^n_{h_{k+1}}}_{H^1(\Omega)}^2, \\
   \text{err}(p,h_k)&:=\tau\sum_{n=1}^{n_T}\norm{p_{h_k}^n-p^n_{h_{k+1}}}_{L^2(\Omega)}^2.
\end{align*}
We consider mesh refinements with the mesh sizes $h_k \approx 2^{-(k+2)}$ for $k=1,\ldots,6$. In \cref{table_10-2} we show the errors and experimental orders of convergence.

\begin{table}[htbp!]
	\centering
	\small
	\caption{Errors and experimental order of convergence (eoc)}
	\begin{tabular}{|c||c|c|c|c|c|c|c|c|c|c|}
		\hline
		$ k $ & $\text{err}(\{\phi_\mA\})$ & eoc & $\text{err}(\{g_\mA\})$ & eoc & $\text{err}(\mathbf{v})$ & eoc & $\text{err}(\nabla\mathbf{v})$ & eoc & $\text{err}(p)$ & eoc \\
		\hline
		2 & $2.61\cdot 10^{-2}$ &     -- & $3.17\cdot 10^{-1}$ & --         & $9.68\cdot 10^{-4}$ & --     & $7.35\cdot 10^{-2}$ & --     & $3.16\cdot 10^{-4}$ & -- \\
		3 & $1.20\cdot 10^{-3}$ & $4.45$ & $4.22\cdot 10^{-1}$ & $-0.41$    & $6.85\cdot 10^{-4}$ & $0.50$ & $7.23\cdot 10^{-2}$ & $0.02$ & $2.06\cdot 10^{-4}$ & $0.62$ \\
		4 & $2.18\cdot 10^{-4}$ & $2.45$ & $1.38\cdot 10^{-1}$ & $\phm1.61$ & $5.52\cdot 10^{-5}$ & $3.63$ & $9.38\cdot 10^{-3}$ & $2.95$ & $1.72\cdot 10^{-5}$ & $3.58$ \\
		5 & $1.43\cdot 10^{-5}$ & $3.92$ & $3.36\cdot 10^{-2}$ & $\phm2.04$ & $3.09\cdot 10^{-6}$ & $4.16$ & $6.61\cdot 10^{-4}$ & $3.83$ & $1.11\cdot 10^{-6}$ & $3.95$ \\
		6 & $9.19\cdot 10^{-7}$ & $3.96$ & $8.12\cdot 10^{-3}$ & $\phm2.05$ & $1.91\cdot 10^{-7}$ & $4.01$ & $3.71\cdot 10^{-5}$ & $4.15$ & $6.81\cdot 10^{-8}$ & $4.02$ \\
		\hline
	\end{tabular}
	\label{table_10-2}
\end{table}

We observe that the obtained rates are in perfect agreement with the employed finite element spaces up to the $L^2$ norm of the velocity. Here approximation properties theoretically allow higher order, which is polluted here by the coupling terms with lower approximation rate.

\subsection{Phase separation test}\label{subsec:phasesep}
As a next test case we consider a phase separation example. In such cases, the evolution of concentration fields is governed by diffusion driven by chemical potential differences, leading to the formation of distinct phases and interfaces. A three-phase model captures a richer interplay between all phases, allowing investigation of ternary mixture separation, morphology development, and stability. For simplicity we will consider constant specific densities, i.e. $\rho_\mA=1$. The effects of unmatched specific densities for $N=2$ are considered in \cite{brunk2026simple} and leads to slower evolutions of the denser phases. From modeling perspective this means the model reduces to the incompressible $N$-phase NSCH model.
We choose the final time $T=5$, viscosities $\eta_\mA=0.1$. The initial data is inspired by \cite{Yang21} and given by
\begin{align}
    \psi_\mA = 0.5 + \mathcal{U}(x,y), \qquad \phi_\mA = \frac{\psi_\mA(x,y)}{\sum_\mB \psi_\mB(x,y)}, \qquad \vv_0(x,y)=0,
\end{align}
where $\mathcal{U}(x,y)$ samples the uniform distribution such that $\mathcal{U}(x,y)\in[-0.0025,0.0025]$ with zero mean. The capillary matrix is chosen as in the convergence test and amounts to three different surface tensions. This amounts to a mixed initial state with small perturbations. 
In Figure \ref{fig: phase sep 3} we show snapshots of the temporal evolution. For $N=2$ the result of such an experiment is fast droplet formation and then slow movement due to diffusion. In the binary ($N = 2$) case, phase separation is driven by the minimization of a single interfacial energy. As a result, the system reduces its total interface by forming rounded droplets, whose shape minimizes curvature and surface tension effects. In contrast, the ternary ($N = 3$) system involves three competing phases with different pairwise interfacial energies. Instead of forming isolated droplets, the system must balance multiple interfaces simultaneously, which promotes the formation of interconnected structures and triple junctions. Consequently, droplet formation is less favorable in $N = 3$, as it would increase the total interfacial energy with the other phases. As visualized in  Figure \ref{fig: phase sep 3 meta} the energy is non-increasing over time and the partial volume conservation holds up to machine precision. Since all related total volume and mass conservation laws arise from linear combinations of these, they are also at the level of machine precision, as predicted by the numerical scheme.

\begin{figure}[!ht]
\captionsetup[subfigure]{justification=centering}
\begin{subfigure}{0.9\textwidth}
        \centering
        \includegraphics[trim=0cm 66cm 20cm 00cm, clip,width=0.9\textwidth]{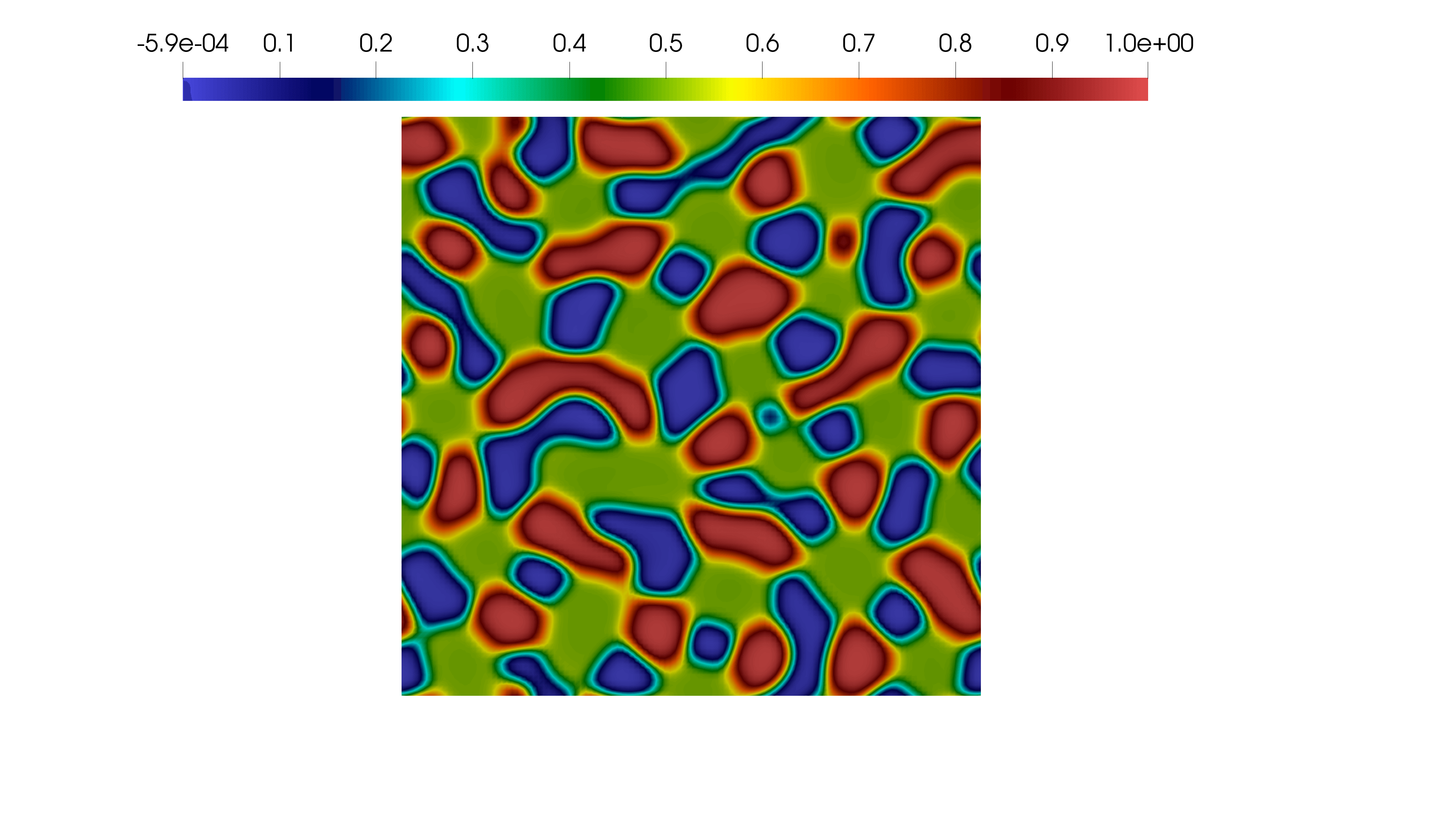}
    \end{subfigure}
\begin{subfigure}{0.196\textwidth}
\centering
\includegraphics[trim=40cm 10cm 40cm 10cm, clip,width=1\textwidth]{figures/phasesep/phi_merge.0100.png}
\caption{$t=0.1$}
\end{subfigure}
\begin{subfigure}{0.196\textwidth}
\centering
\includegraphics[trim=40cm 10cm 40cm 10cm, clip,width=1\textwidth]{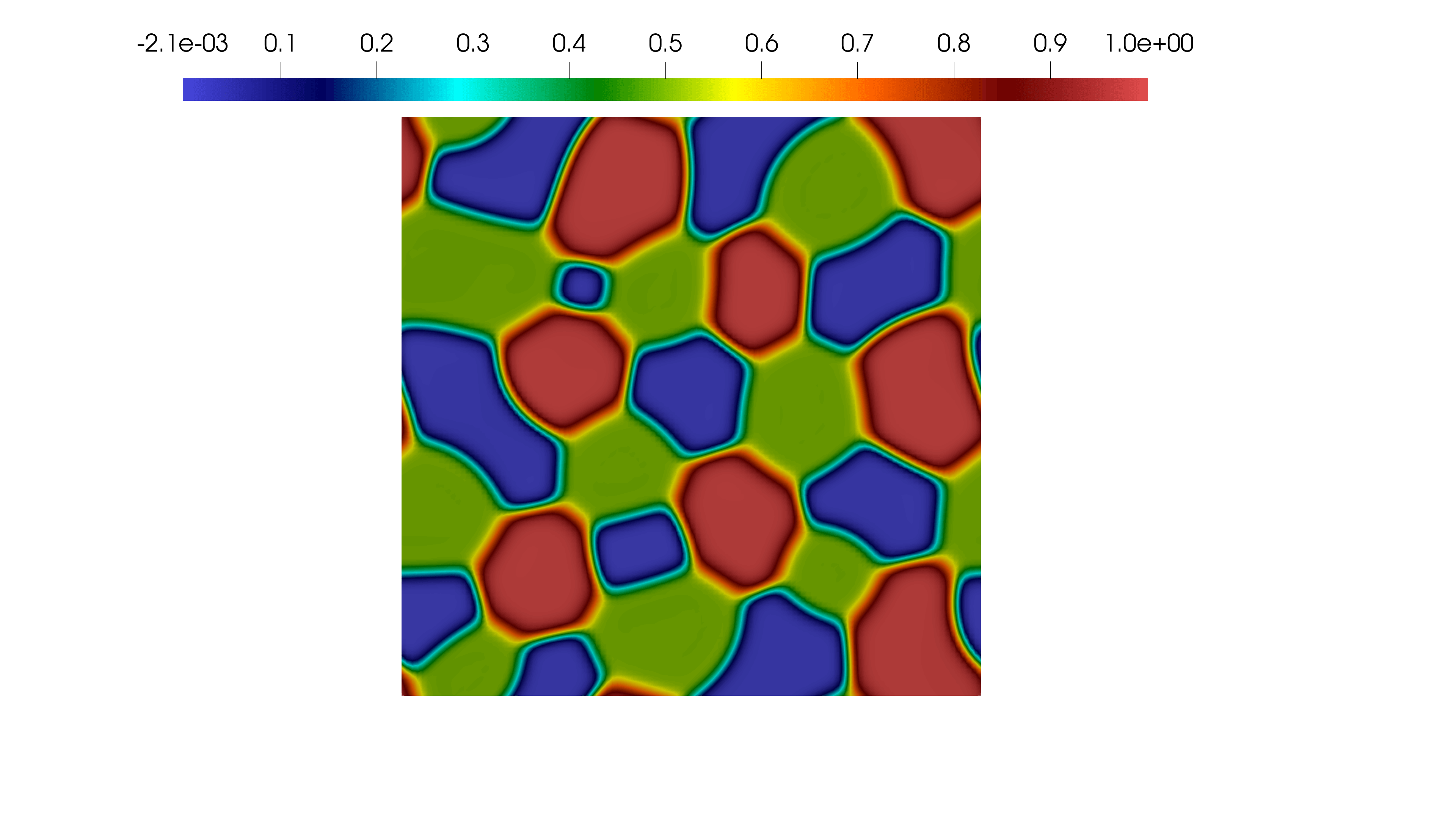}
\caption{$t=1.0$}
\end{subfigure}
\begin{subfigure}{0.196\textwidth}
\centering
\includegraphics[trim=40cm 10cm 40cm 10cm, clip,width=1\textwidth]{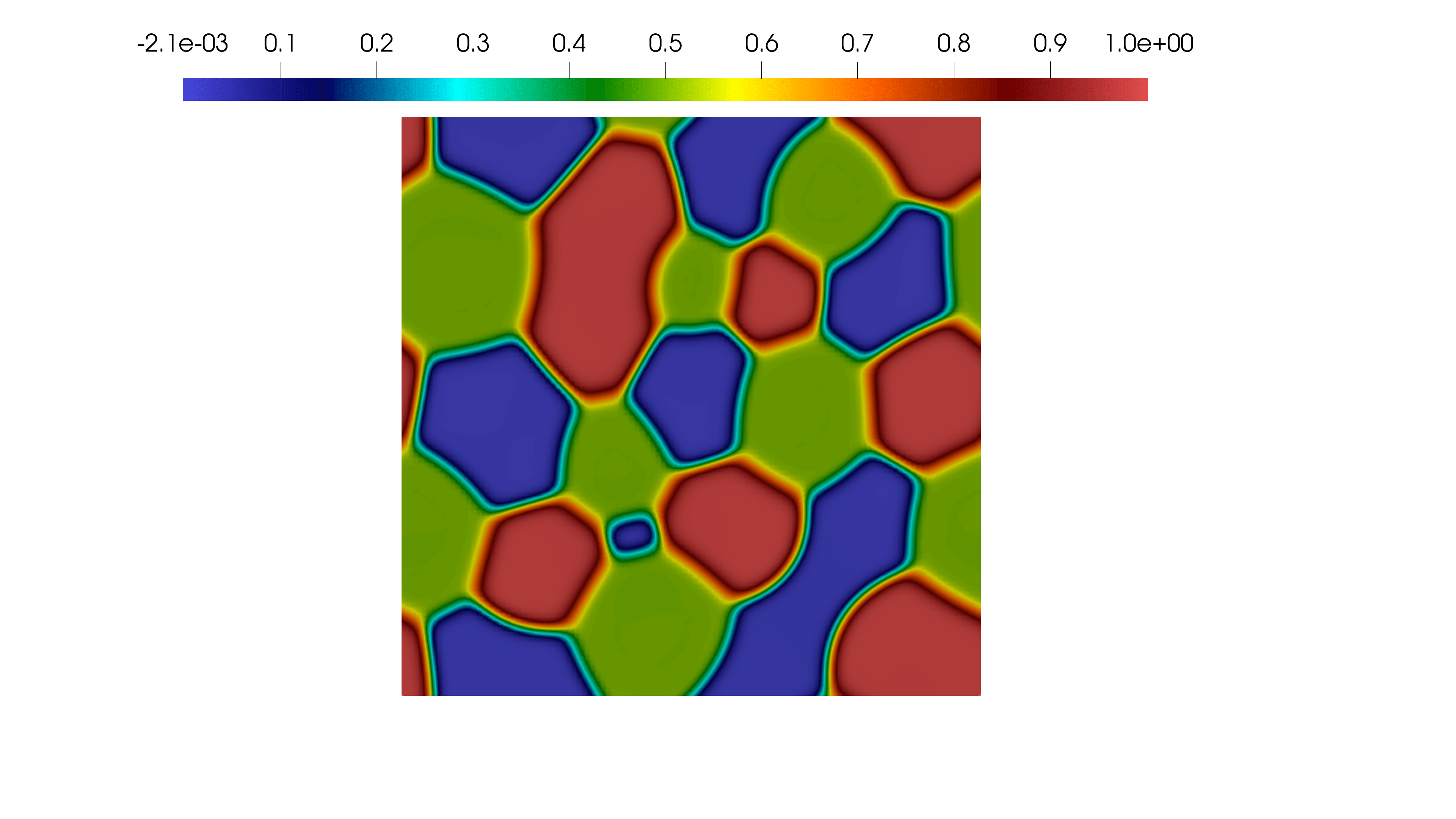}
\caption{$t=2.5$}
\end{subfigure}
\begin{subfigure}{0.196\textwidth}
\centering
\includegraphics[trim=40cm 10cm 40cm 10cm, clip,width=1\textwidth]{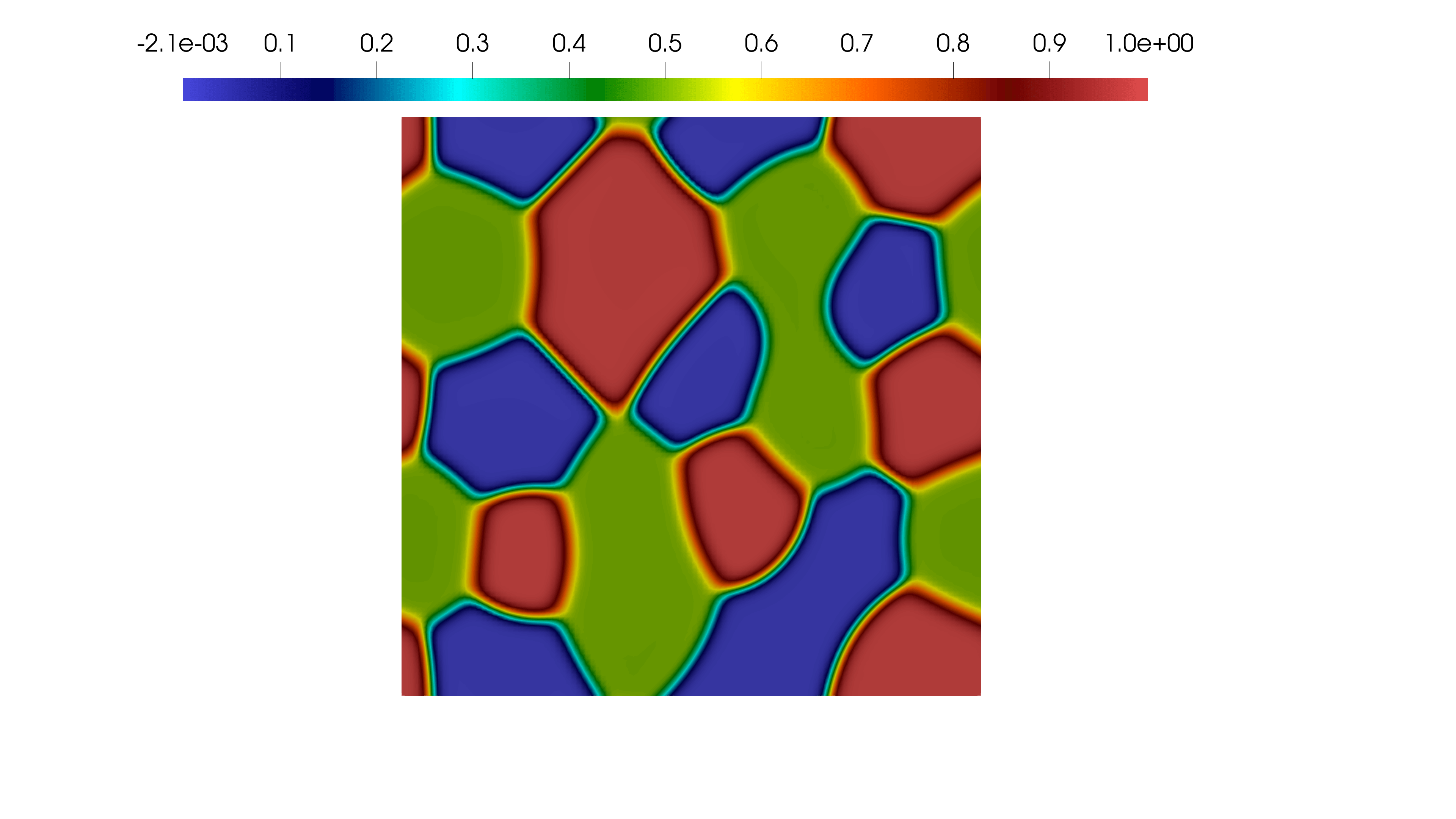}
\caption{$t=4$}
\end{subfigure}
\begin{subfigure}{0.196\textwidth}
\centering
\includegraphics[trim=40cm 10cm 40cm 10cm, clip,width=1\textwidth]{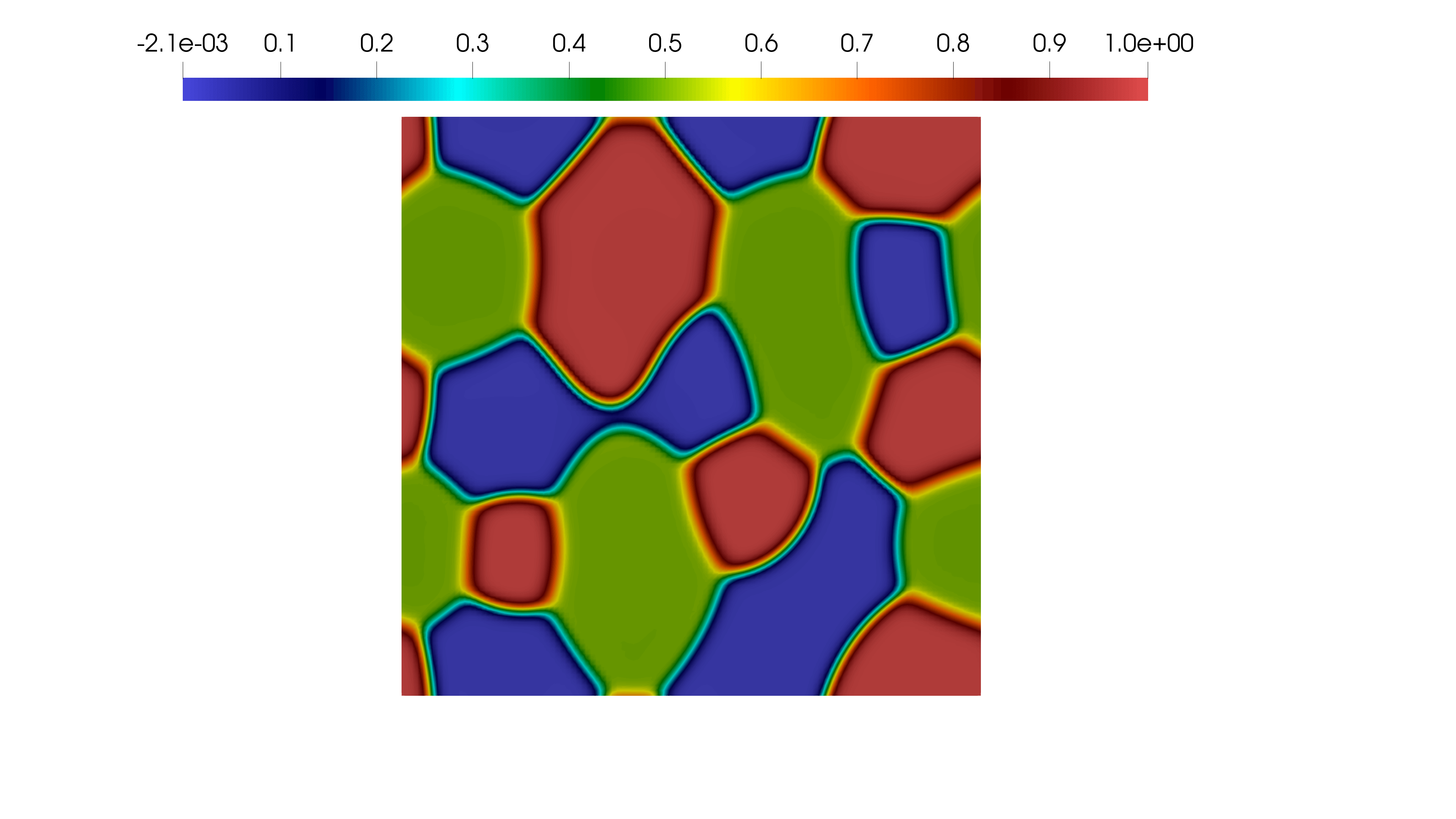}
\caption{$t=5$}
\end{subfigure}
\caption{Phase separation test. Temporal evolution of the phase field $\phi_1+0.5\phi_2$.}
\label{fig: phase sep 3}
\end{figure}

\begin{figure}[!ht]
\captionsetup[subfigure]{justification=centering}
\begin{subfigure}{0.49\textwidth}
\centering
\includegraphics[width=1\textwidth]{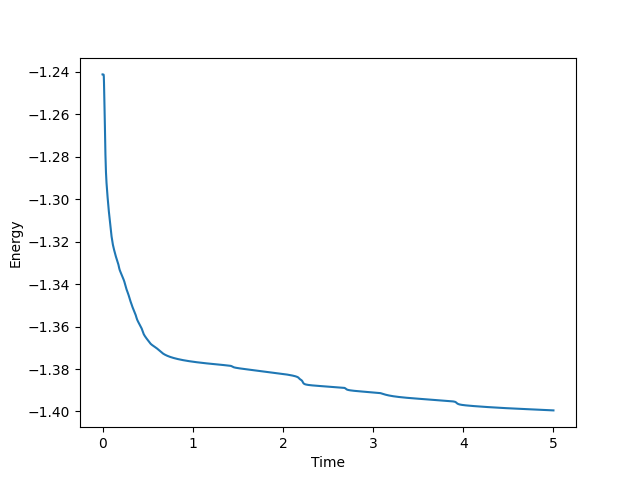}
\caption{Energy dissipation over time}
\end{subfigure}
\begin{subfigure}{0.49\textwidth}
\centering
\includegraphics[width=1\textwidth]{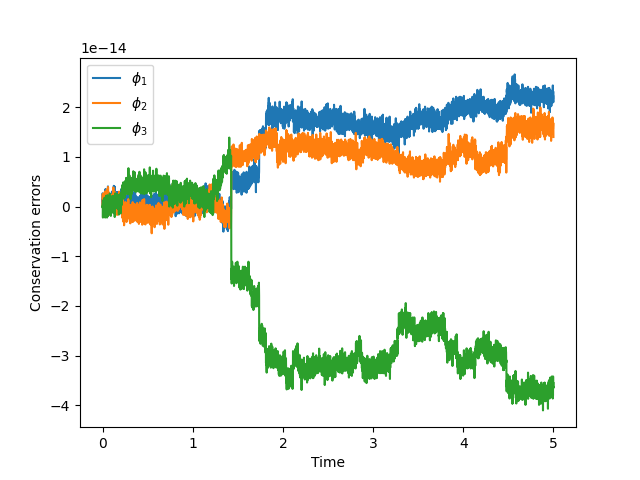}
\caption{Mass conservation error over time}
\end{subfigure}
\caption{Phase separation test. Energy dissipation and mass conservation}\label{fig: phase sep 3 meta}
\end{figure}

\newpage
\subsection{Two-fluid rising bubble test cases}\label{subsec:risingbubble}

\begin{wrapfigure}[20]{r}{0.4\textwidth}
\begin{center}
\begin{tikzpicture}
    % Draw the gray background rectangle
    \fill[fill=red!5,draw=black] (0, 0) rectangle (3, 6);
    
    % Draw the white circle with a black border
    \draw[fill=blue!5,draw=black] (1.5, 1.5) circle (0.75);
    
    % Add arrows to indicate dimensions
    \draw[<->] (0, -0.5) -- (3, -0.5) node[midway,below] {$1$};
    \draw[<->] (-0.5, 0) -- (-0.5, 6) node[midway,left] {$2$};
    \draw[<->] (3.5, 0.75) -- (3.5, 2.25) node[midway,right] {$D_0 = 0.5$};
    \draw[-] (1.5, 0.75) -- (3.5, 0.75);
    \draw[-] (1.5, 2.25) -- (3.5, 2.25);
    \draw[<->] (1.5, 0.0) -- (1.5, 1.5) node[midway,below left] {$0.5$};
    \draw[-] (1.5, 2.25) -- (3.5, 2.25);
    \node at (1.5, 6.3) {$v_1 = v_2 = 0$}; 
    \node at (1.5, -0.3) {$v_1 = v_2 = 0$};
    \node at (-1.2, 4.0) {$v_1 = 0$}; 
    \node at (3.7, 4.0) {$v_1 = 0$}; 
    \node at (1.5, 1.8) {fluid 2};
    \node at (1.5, 3.3) {fluid 1}; 
\end{tikzpicture}
    \caption{Rising bubble: Schematic representation of the problem setup}
    \label{fig:sketch 2D rising bubble problem}
\end{center}
\end{wrapfigure}
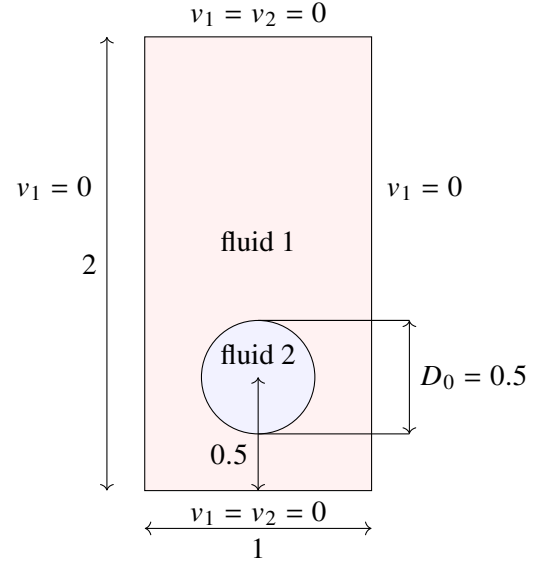
We consider the classical two-dimensional rising bubble benchmark of
Hysing et al.~\cite{hysing2009quantitative}. The computational domain is
$\Omega=[0,1]\times[0,2]$. Initially, a circular bubble of the lighter
fluid is placed in the lower part of the column, with diameter $D_0=2R_0=0.5$ and center $\mathbf{x}_c=(0.5,0.5)$. The surrounding
fluid is the heavier phase. The boundary conditions are no-penetration
on the vertical walls and no-slip on the horizontal walls, i.e.,
\begin{align}
    \mathbf{v}\cdot\mathbf{n}=&~0
    \quad\text{on } \{0,1\}\times[0,2],\nn\\
    \qquad
    \mathbf{v}=&~\mathbf{0}
    \quad\text{on } [0,1]\times\{0,2\}.
\end{align}
A schematic representation of the configuration is shown in
\cref{fig:sketch 2D rising bubble problem}.

Simulations were conducted on a uniform rectangular mesh. We simulate half of the domain and apply symmetry boundary conditions along the vertical line $x=1/2$. The time step size is set as $\Delta t_n = 0.128 h $, while $\varepsilon = h$. The benchmark problem consists of two cases, each defined by different parameter values, as listed in \cref{table: parameters 2D RB cases}. \\

\begin{table}[htbp]
\centering
\begin{tabularx}{\textwidth}{XXXXXXXXX}
%\hline\\[-6pt]
Case & \hspace{0.1cm} $\rho_1$ & \hspace{0.1cm} $\rho_2$ & $\nu_1$ & $\nu_2$ & \hspace{0.1cm} $\gamma_{12}$ & \hspace{0.1cm} $g$ & $\mathbb{A}{\rm r}$ & $\mathbb{E}{\rm o}$ \\[4pt]
\hline\\[-6pt]
\hspace{0.5cm}1 & $1000$ & $100$ & $10$ & $1$   & $24.5$ & $0.98$ & $35$ & \hspace{0.05cm} $10$   \\[6pt]
\hspace{0.5cm}2 & $1000$ & \hspace{0.1cm} $1$   &  $10$  & $0.1$ &$1.96$  & $0.98$ & $35$ & $125$  \\[6pt]
\hline
\end{tabularx}
\caption{Parameters for the two-dimensional rising bubble cases ($N=2)$.}
\label{table: parameters 2D RB cases}
\end{table}

The calibrated capillarity matrices, interface widths, and scale factors (for details, see \cite{surfacetension2026}) take the form
\begin{subequations}
    \begin{align} 
   & \text{Case 1}: && \text{Case 2}:\nn\\
    \bar{\boldsymbol{\kappa}} =&~{\small 10^{2}\begin{pmatrix}
4.0209152514 &-8.0418305027\\
-8.0418305027 &4.0209152514
\end{pmatrix}}, &  \bar{\boldsymbol{\kappa}} =&~{\small \begin{pmatrix}
2.5703710049 &-5.1407420098\\
-5.1407420098 &2.5703710049
\end{pmatrix}} \\
\varepsilon_{12} = &~24.3042511,  &  \varepsilon_{12} =&~ 1.94308232 \\
\varepsilon_0 = &~ 3.2144625 \times 10^{-4}, & \varepsilon_0 =&~  4.02066622 \times 10^{-3}
\end{align}
\end{subequations}

%-0.001 1.001 -0.001 2.001

We show in \cref{fig: case 1 phi,fig: case 2 phi} the evolution of the rising bubble (using mesh width $h=1/128$) for cases 1 and 2, respectively. In Case 1, the bubble remains close to its initial shape and only undergoes mild deformation. On the other hand, Case 2 shows significant shape changes due to relatively stronger inertial effects. 

\begin{figure}[!ht]
\captionsetup[subfigure]{justification=centering}
\begin{subfigure}{0.16\textwidth}
\centering
\includegraphics[width=1\textwidth]{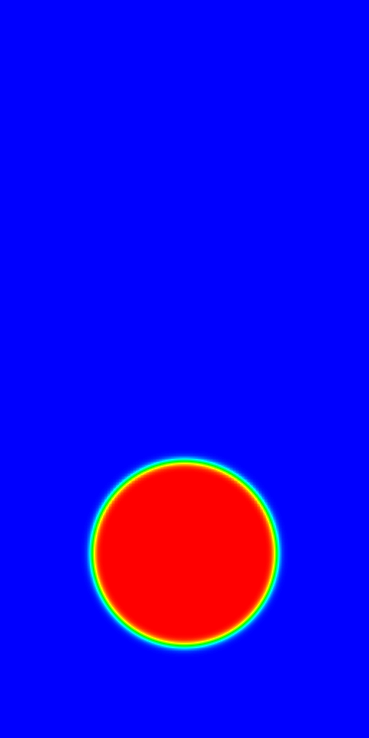}
\caption{$t=0.0$}
\end{subfigure}
\begin{subfigure}{0.16\textwidth}
\centering
\includegraphics[width=1\textwidth]{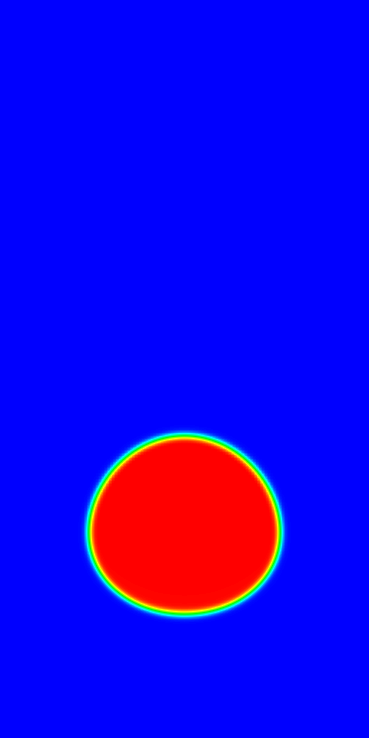}
\caption{$t=0.6$}
\end{subfigure}
\begin{subfigure}{0.16\textwidth}
\centering
\includegraphics[width=1\textwidth]{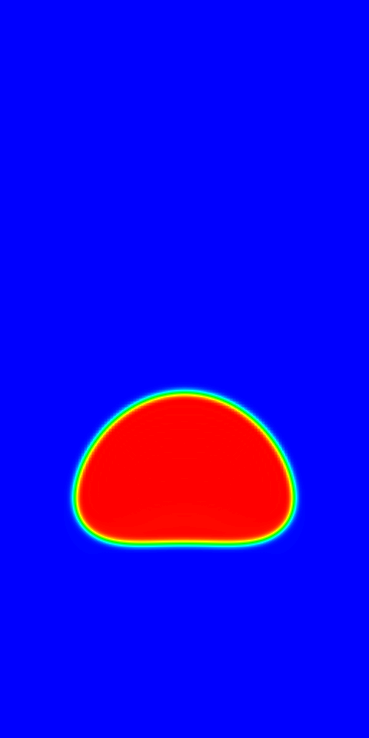}
\caption{$t=1.2$}
\end{subfigure}
\begin{subfigure}{0.16\textwidth}
\centering
\includegraphics[width=1\textwidth]{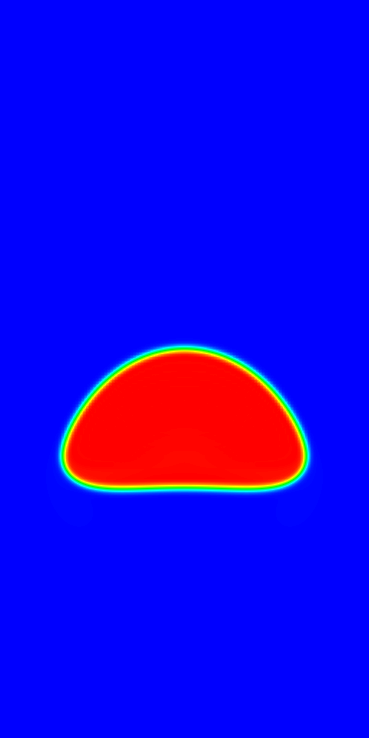}
\caption{$t=1.8$}
\end{subfigure}
\begin{subfigure}{0.16\textwidth}
\centering
\includegraphics[width=1\textwidth]{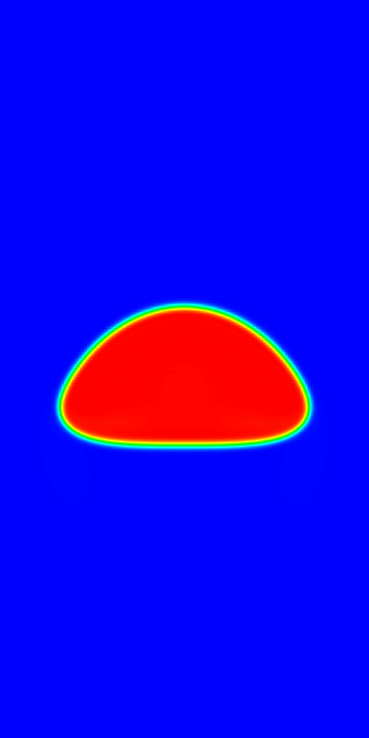}
\caption{$t=2.4$}
\end{subfigure}
\begin{subfigure}{0.16\textwidth}
\centering
\includegraphics[width=1\textwidth]{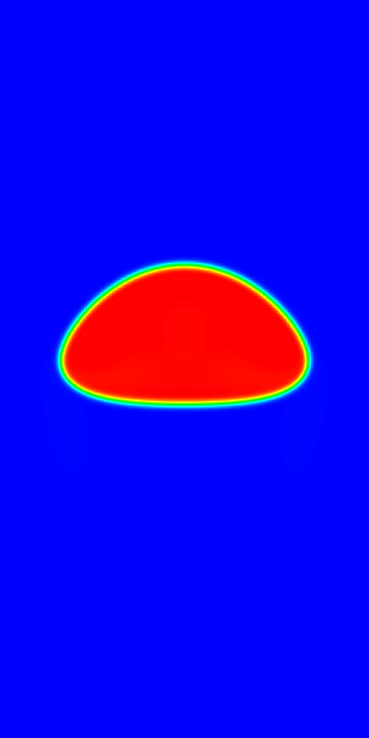}
\caption{$t=3.0$}
\end{subfigure}
\caption{Case 1. Visualization of the phase field $\phi_2$ for element size $h =1/128$.}
\label{fig: case 1 phi}
\end{figure}

\begin{figure}[!ht]
\captionsetup[subfigure]{justification=centering}
\begin{subfigure}{0.16\textwidth}
\centering
\includegraphics[width=1\textwidth]{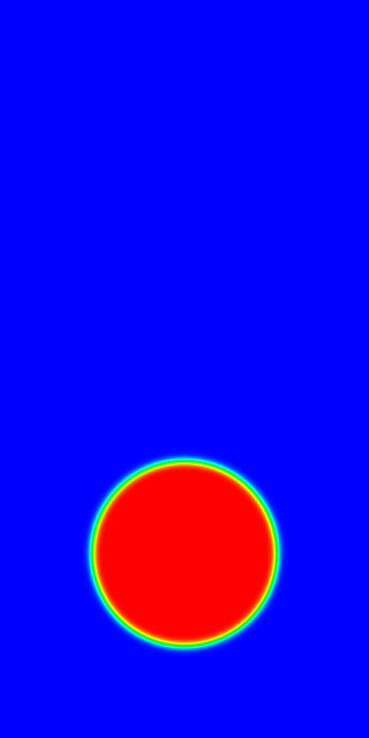}
\caption{$t=0.0$}
\end{subfigure}
\begin{subfigure}{0.16\textwidth}
\centering
\includegraphics[width=1\textwidth]{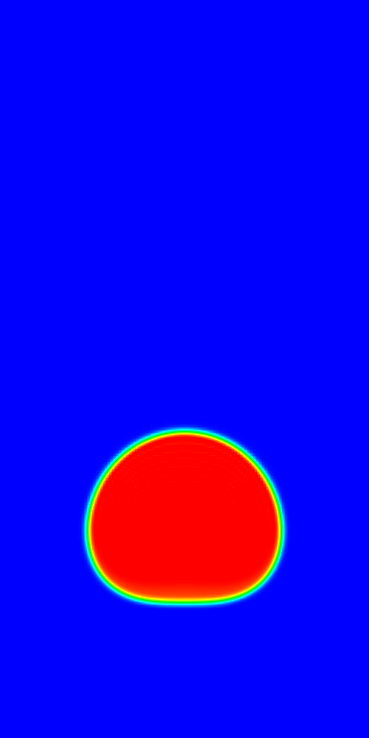}
\caption{$t=0.6$}
\end{subfigure}
\begin{subfigure}{0.16\textwidth}
\centering
\includegraphics[width=1\textwidth]{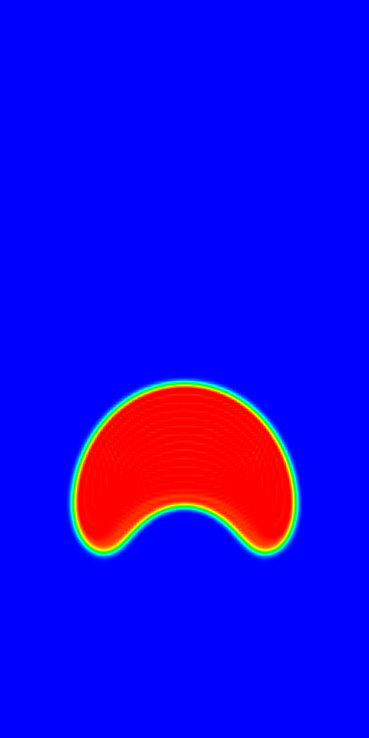}
\caption{$t=1.2$}
\end{subfigure}
\begin{subfigure}{0.16\textwidth}
\centering
\includegraphics[width=1\textwidth]{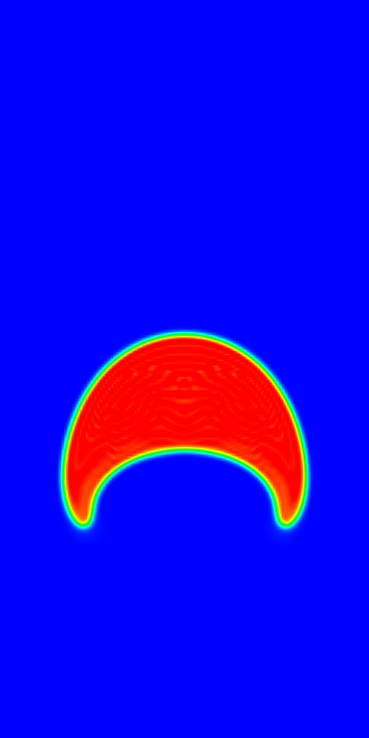}
\caption{$t=1.8$}
\end{subfigure}
\begin{subfigure}{0.16\textwidth}
\centering
\includegraphics[width=1\textwidth]{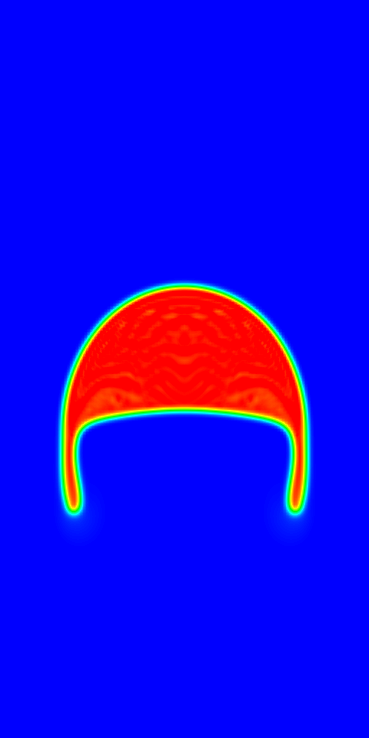}
\caption{$t=2.4$}
\end{subfigure}
\begin{subfigure}{0.16\textwidth}
\centering
\includegraphics[width=1\textwidth]{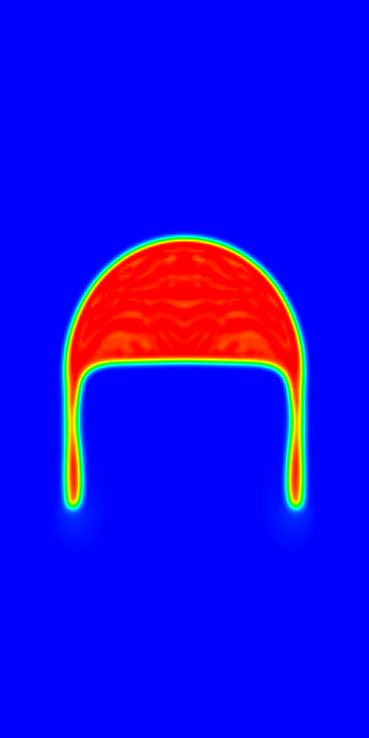}
\caption{$t=3.0$}
\end{subfigure}
\caption{Case 2. Visualization of the phase field $\phi_2$ for element size $h =1/128$.}
\label{fig: case 2 phi}
\end{figure}

\newpage
Next, we verify conservation of the volume fractions and energy dissipation for both cases in \cref{fig: Volume,fig: Energy}, respectively. We observe that the total volume fractions are preserved throughout the simulations for all element sizes $h=1/32, 1/64, 1/128$. The errors of the volume fractions are zero up to machine precision. Finally, the energy evolution plots confirm the energy-dissipative structure of the method. These results are compatible with \cref{thm:discreten_stable}.

\begin{figure}[!ht]
\begin{subfigure}{0.49\textwidth}
\centering
\includegraphics[width=0.95\textwidth]{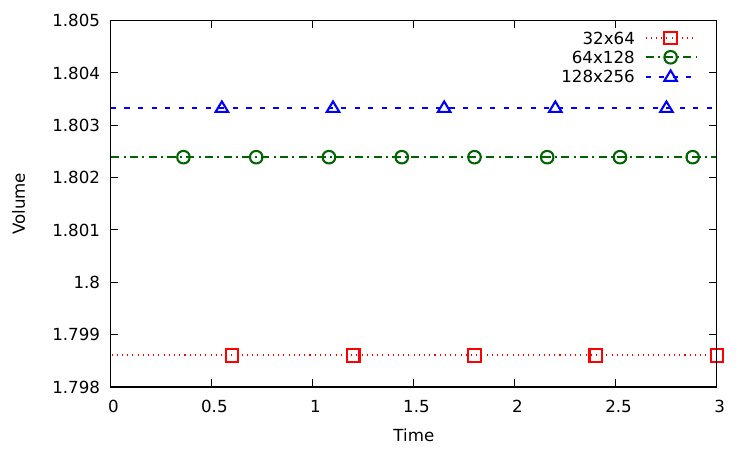}
\caption{Volume fraction evolution fluid 1 - Case 1.}
\end{subfigure}
\begin{subfigure}{0.49\textwidth}
\centering
\includegraphics[width=0.95\textwidth]{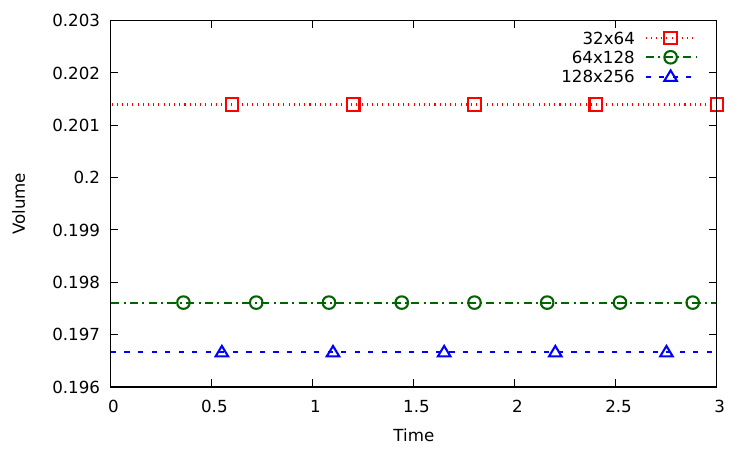}
\caption{Volume fraction evolution fluid 2 - Case 1.}
\end{subfigure}
\begin{subfigure}{0.49\textwidth}
\centering
\includegraphics[width=0.95\textwidth]{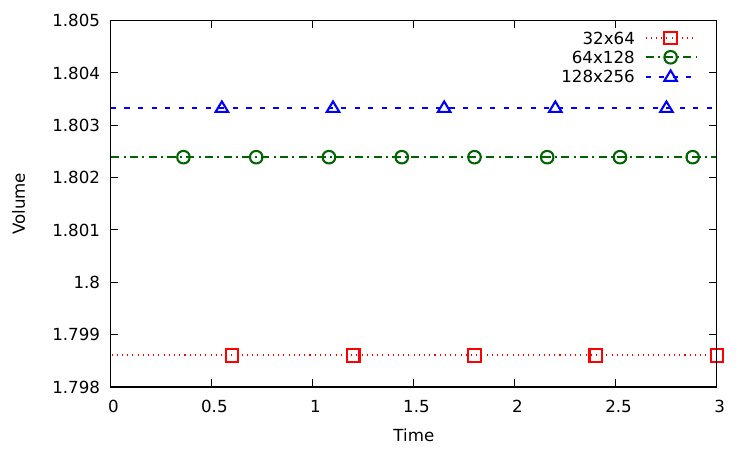}
\caption{Volume fraction evolution fluid 1 - Case 2.}
\end{subfigure}
\begin{subfigure}{0.49\textwidth}
\centering
\includegraphics[width=0.95\textwidth]{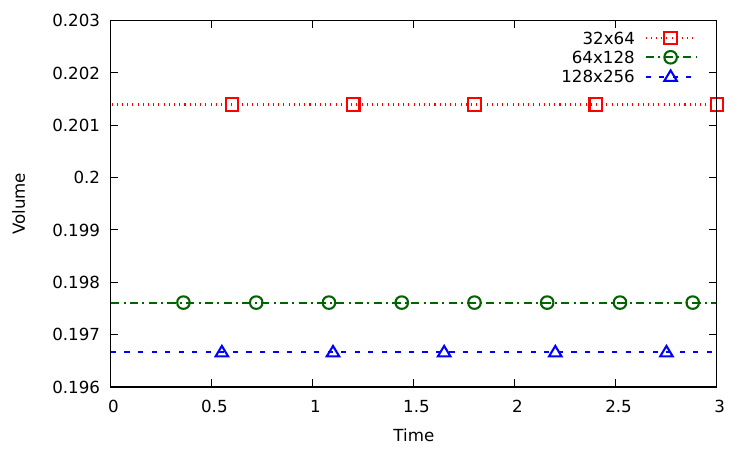}
\caption{Volume fraction evolution fluid 2 - Case 2.}
\end{subfigure}
\centering
\begin{subfigure}{0.49\textwidth}
\centering
\includegraphics[width=0.95\textwidth]{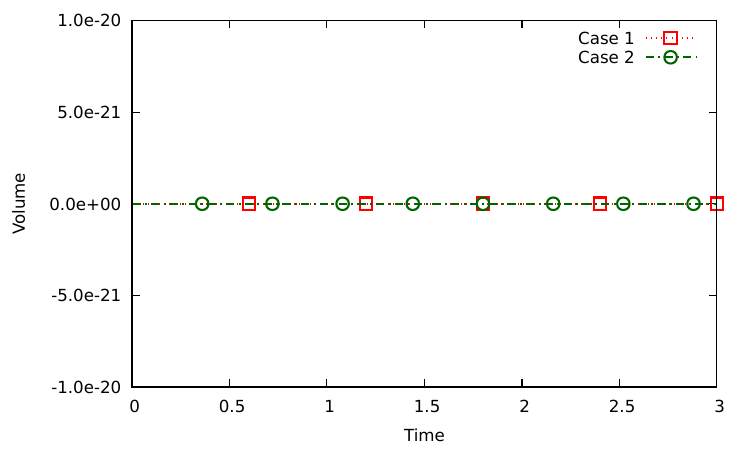}
\caption{Volume fraction evolution error fluid 1. Cases 1 and 2.}
\end{subfigure}
\begin{subfigure}{0.49\textwidth}
\centering
\includegraphics[width=0.95\textwidth]{figures/Volume_phase1_128.pdf}
\caption{Volume fraction evolution error fluid 2. Cases 1 and 2.}
\end{subfigure}
\caption{Evolution of the volume conservation (error); (a), (b) Case 1 for $h = 1/32, 1/64, 1/128$, (c), (d) Case 2 for $h = 1/32, 1/64, 1/128$, (e),(f) Error for $h = 1/128$ for fluid 1 and 2.}
\label{fig: Volume}
\end{figure}

\begin{figure}[!ht]
\begin{subfigure}{0.49\textwidth}
\centering
\includegraphics[width=0.95\textwidth]{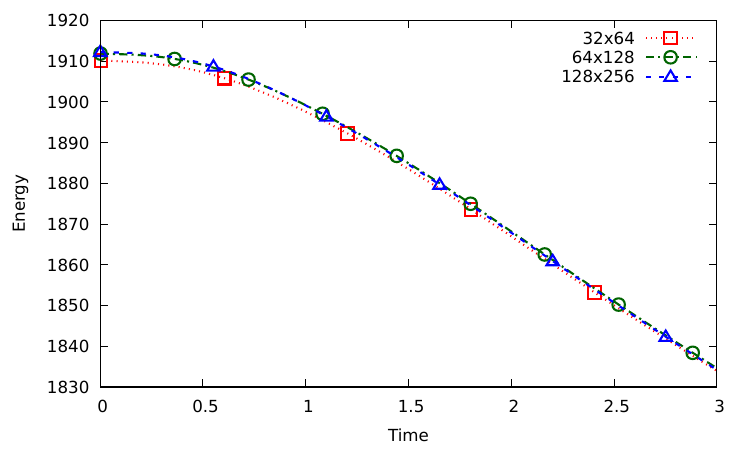}
\caption{Case 1.}
\end{subfigure}
\begin{subfigure}{0.49\textwidth}
\centering
\includegraphics[width=0.95\textwidth]{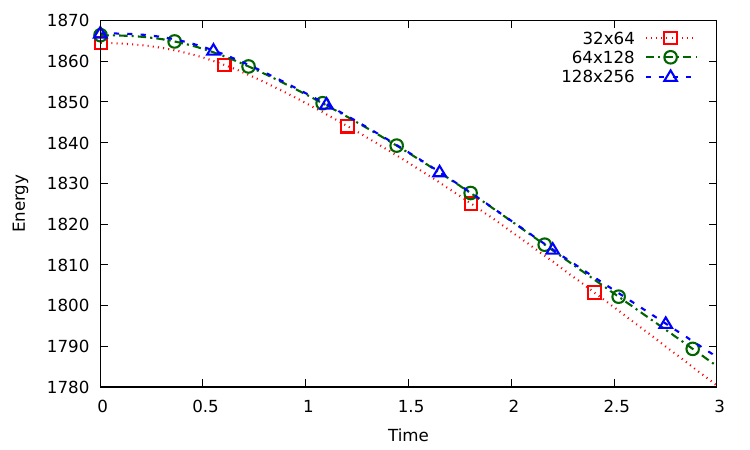}
\caption{Case 2.}
\end{subfigure}
\caption{Evolution of the energy $E$ for different mesh widths $h = 1/32, 1/64, 1/128$.}
\label{fig: Energy}
\end{figure}
\clearpage
To compare the simulations quantitatively with benchmark data, we track the bubble center of mass and rise velocity. These quantities are respectively defined by
    \begin{align}
        y_b := \dfrac{\int_{\phi_2 > 1/2} y ~{\rm d}x}{\int_{\phi_2 > 1/2} ~ {\rm d}x},\qquad
        v_b := \dfrac{\int_{\phi_2 > 1/2} v_2 ~{\rm d}x}{\int_{\phi_2 > 1/2} ~ {\rm d}x},
    \end{align}
where $\{\phi<0\}$ denotes the bubble region.

In \cref{fig: case 1 CoM,fig: case 1 RV,fig: case 2 CoM,fig: case 2 RV}, we show for different mesh sizes the center of mass and rise velocity for both cases. We compare these results with reference data from the literature. These are the simulations performed using the TP2D, FreeLIFE, and MooNMD codes \cite{hysing2009quantitative}, as well as the (single-phase field) NSCH models proposed by Abels et al. \cite{abels2012thermodynamically}, Boyer \cite{boyer2002theoretical}, and Ding et al. \cite{ding2007diffuse}. These (single phase-field) NSCH simulations were conducted by Aland and Voigt \cite{aland2012benchmark}. Furthermore, we compare with the volume-averaged and mass-averaged velocity formulations of the (single-phase field) NSCH model provided by \cite{ten2024divergence} and \cite{brunk2026simple}, respectively. In the figures we denote the single phase-field volume-averaged and mass-averaged velocity formulations as `$\text{NSCH}_{\text{vol}}$', and `$\text{NSCH}_{\text{mass}}$' respectively, and we denote the current computations by `$\text{NSCH}_{\text{mix}}$'.

\begin{figure}[!ht]
\begin{subfigure}{0.49\textwidth}
\centering
\includegraphics[width=0.95\textwidth]{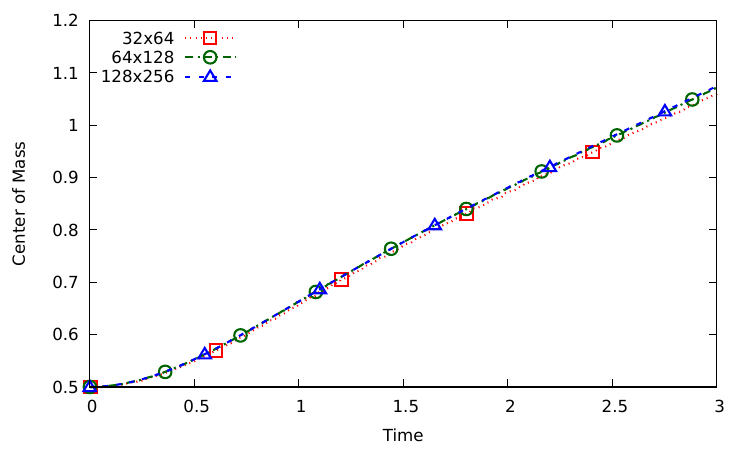}
\caption{$h = 1/32, 1/64, 1/128$.}
\end{subfigure}
\begin{subfigure}{0.49\textwidth}
\centering
\includegraphics[width=0.95\textwidth]{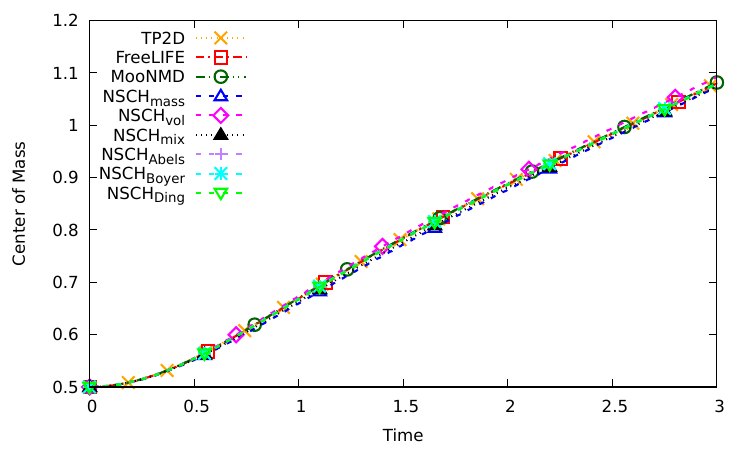}
\caption{Comparison with data from the literature.}
\end{subfigure}
\caption{Case 1. Center of mass (a) for different mesh sizes, and (b) a comparison of the finest mesh results to reference data.}
\label{fig: case 1 CoM}
\end{figure}

\begin{figure}[!ht]
\begin{subfigure}{0.49\textwidth}
\centering
\includegraphics[width=0.95\textwidth]{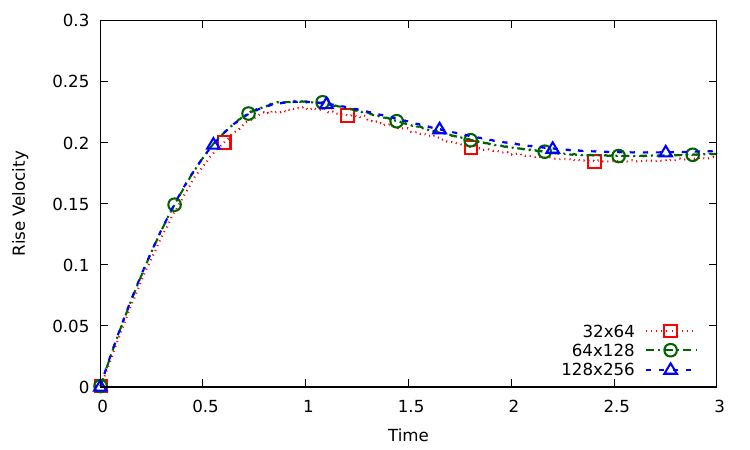}
\caption{$h = 1/32, 1/64, 1/128$.}
\end{subfigure}
\begin{subfigure}{0.49\textwidth}
\centering
\includegraphics[width=0.95\textwidth]{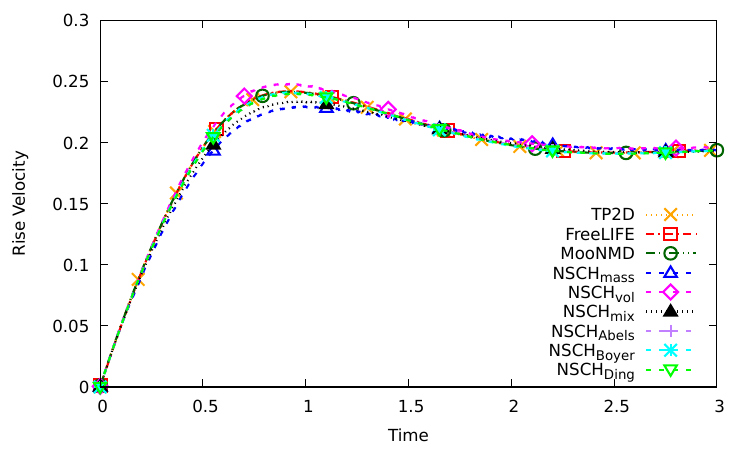}
\caption{Comparison with data from the literature.}
\end{subfigure}
\caption{Case 1. Rise velocity (a) for different mesh sizes, and (b) a comparison of the finest mesh results to reference data.}
\label{fig: case 1 RV}
\end{figure}

\begin{figure}[!ht]
\begin{subfigure}{0.49\textwidth}
\centering
\includegraphics[width=0.95\textwidth]{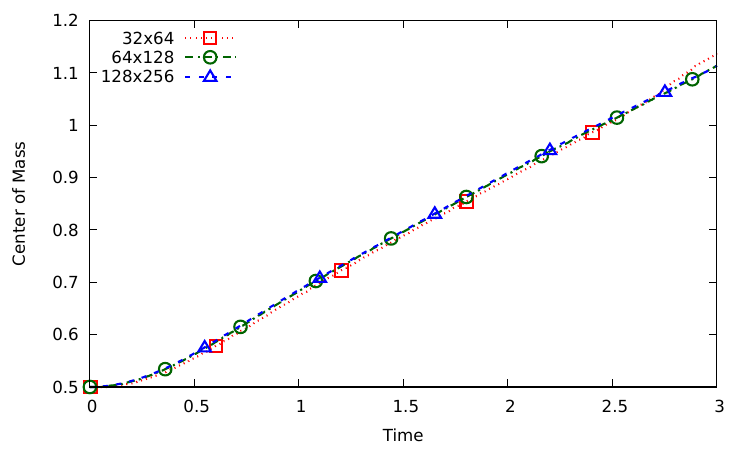}
\caption{$h = 1/16, 1/32, 1/64, 1/128$.}
\end{subfigure}
\begin{subfigure}{0.49\textwidth}
\centering
\includegraphics[width=0.95\textwidth]{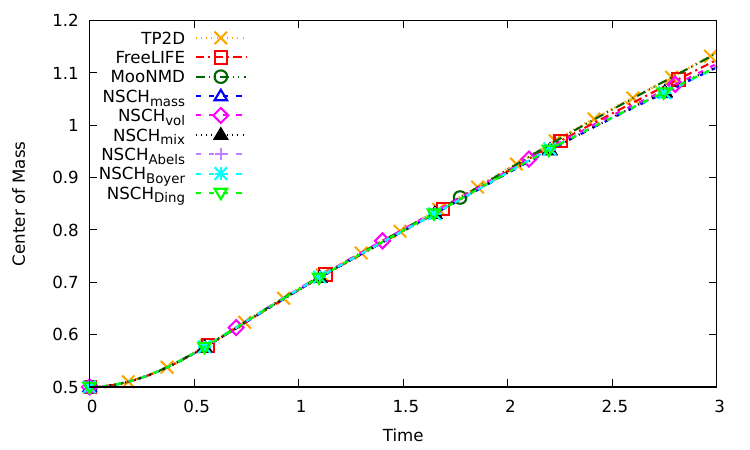}
\caption{Comparison with data from the literature.}
\end{subfigure}
\caption{Case 2. Center of mass (a) for different mesh sizes, and (b) a comparison of the finest mesh results to reference data.}
\label{fig: case 2 CoM}
\end{figure}

\begin{figure}[!ht]
\begin{subfigure}{0.49\textwidth}
\centering
\includegraphics[width=0.95\textwidth]{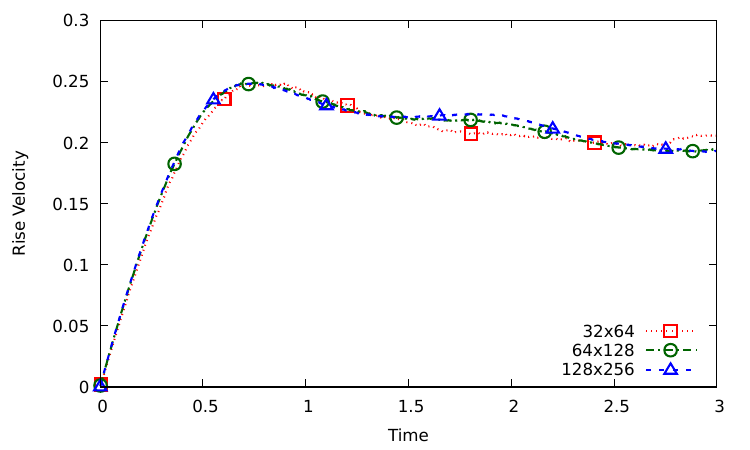}
\caption{$h = 1/32, 1/64, 1/128$.}
\end{subfigure}
\begin{subfigure}{0.49\textwidth}
\centering
\includegraphics[width=0.95\textwidth]{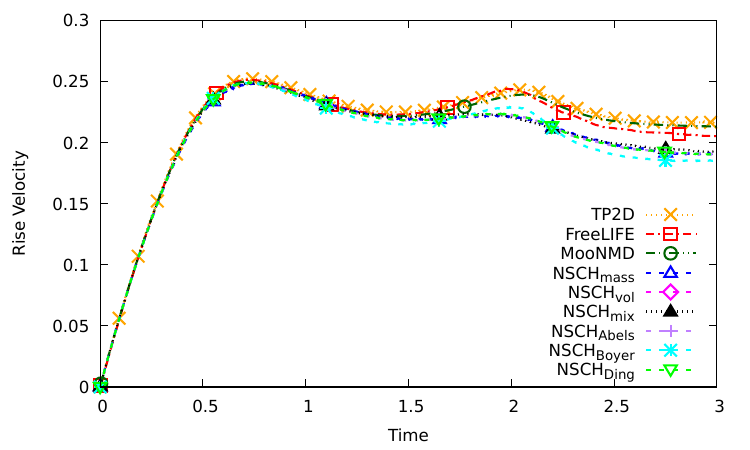}
\caption{Comparison with data from the literature.}
\end{subfigure}
\caption{Case 2. Rise velocity (a) for different mesh sizes, and (b) a comparison of the finest mesh results to reference data.}
\label{fig: case 2 RV}
\end{figure}

For both benchmark cases, the center of mass agrees well with the reference data. The rise velocity also closely follows the reference trends. In Case 2, larger deviations appear for $t>1.5$ when compared with the TP2D, FreeLIFE, and MooNMD data. In this later regime, however, the present results remain close to
the other NSCH computations, indicating that the difference is mainly associated with the differences between sharp-interface and diffuse-interface formulations, rather than the particular NSCH formulation/discretization.

\subsection{Three-fluid problem}\label{subsec:risingbubblethree}

We also consider a two-dimensional version of the three-fluid rising-bubble benchmark. The setup is such that a gas bubble (phase $1$) rises from the lower liquid layer (phase $2$) toward the liquid--liquid interface separating phases $2$ and $3$. 
To reduce the computational cost, we compute only half of the symmetric domain and impose symmetry boundary conditions on the symmetry axis. The half-domain is discretized by $64\times480$ finite elements. We use $\Delta t=h$ and $\varepsilon=h$, where $h$ denotes the mesh width. 

The target surface tensions are prescribed as
\begin{align}
  \gamma_{12}=\gamma_{13}=0.07~\mathrm{N\,m^{-1}},\qquad
  \gamma_{23}=0.05~\mathrm{N\,m^{-1}}.
\end{align}
Using the algorithm provided in \cite{surfacetension2026}, we calibrate the capillary matrix $\boldsymbol{\kappa}$ and interface width parameters to match the target surface tensions and target interface width $\varepsilon$:
\begin{subequations}
    \begin{align}    \bar{\boldsymbol{\kappa}} =&~10^{-3}\begin{pmatrix}
4.1361833328 & -6.5504823549  & -1.7218843107 \\
-6.5504823549  & 6.5504823549
 & -6.5504823549 \\
-1.7218843107  &  -6.5504823549  &  4.1361833328 
\end{pmatrix}\\
\bar{\varepsilon}_{12} = &~0.0819004278, \quad \bar{\varepsilon}_{13} = 0.0581144728, \quad \bar{\varepsilon}_{23} = 0.0819004278, \\
r_1 : &~ \varepsilon_0 =  1.07546309 \times 10^{-3}, \qquad r_2 : \varepsilon_0 = 1.55942257  \times 10^{-3}. 
\end{align}
\end{subequations}
The other material parameters are listed in \cref{table:N3-rising-bubble}.
\begin{table}[!ht]
\centering
\begin{tabular}{lcc}
\hline
phase & density $\rho$ [$\mathrm{kg\,m^{-3}}$] & viscosity $\eta$ [$\mathrm{Pa\,s}$] \\
\hline
bubble ($\phi_1$) & $1$ & $10^{-4}$ \\
heavy liquid ($\phi_2$) & $1200$ & $0.15$ \\
light liquid ($\phi_3$) & $1000$ & $0.10$ \\
\hline
\end{tabular}
\caption{Material parameters for the three-fluid rising-bubble problem.}
\label{table:N3-rising-bubble}
\end{table}
Gravity acts in the negative vertical direction, $ \mathbf g=-g\mathbf e_y$, $g=9.81~\mathrm{m\,s^{-2}}$. On the outer boundary we impose no-penetration conditions for the velocity. For the phase fields and chemical potentials, homogeneous Neumann conditions are
used.

In this two-dimensional setting, the critical size for penetration is estimated by balancing buoyancy against the capillary force at the liquid--liquid interface. For the parameters used here, the critical radius is $r_p\approx 1.8\times10^{-3}\,{\rm m}$. The resulting critical radius for the two-dimensional case is lower than that of the axisymmetric case (see e.g. \cite{surfacetension2026}).

\begin{remark}[Derivation critical radius for penetration]
For a circular bubble of radius $r$ and area $A=\pi r^2$, the buoyancy force per unit out-of-plane length is estimated by $    F_b=(\rho_3-\rho_1)gA$. The capillary force opposing penetration is estimated by $F_\gamma=2\gamma_{23}$, because the liquid--liquid interface meets the bubble in two points in the two-dimensional cross-section. The critical area is obtained from $F_b=F_\gamma$, which provides $A_p=\pi r_p^2 = (2\gamma_{23})/((\rho_3-\rho_1)g)$ so that the corresponding critical radius is 
\begin{align}\label{eq:rp_criterion}
    r_p  = \sqrt{ \frac{2\gamma_{23}}{\pi(\rho_3-\rho_1)g} }.
\end{align}
\end{remark}

We consider two bubble radii $ r_1=1.5\times10^{-3}\,{\rm m}$ and $ r_2=2.0\times10^{-3}\,{\rm m}$. The first radius satisfies $r_1<r_p$, so the bubble is expected to remain trapped at the liquid--liquid interface. The second radius satisfies $r_2>r_p$, and the criterion predicts penetration into the upper liquid layer.

\begin{figure}[!ht]
\captionsetup[subfigure]{justification=centering}\begin{center}
\begin{subfigure}{0.12\textwidth}
\centering
\includegraphics[width=1\textwidth]{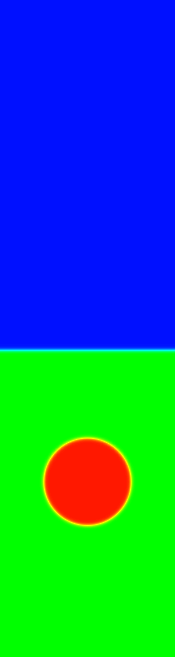}
\caption{$t=0.00$}
\end{subfigure}
\begin{subfigure}{0.12\textwidth}
\centering
\includegraphics[width=1\textwidth]{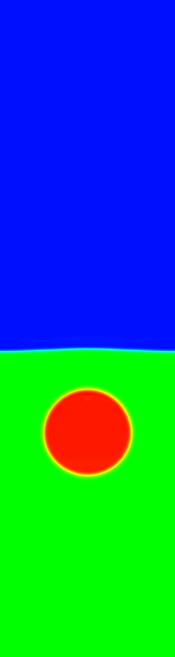}
\caption{$t=0.10$}
\end{subfigure}
\begin{subfigure}{0.12\textwidth}
\centering
\includegraphics[width=1\textwidth]{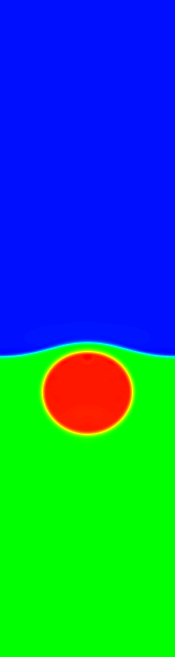}
\caption{$t=0.20$}
\end{subfigure}
\begin{subfigure}{0.12\textwidth}
\centering
\includegraphics[width=1\textwidth]{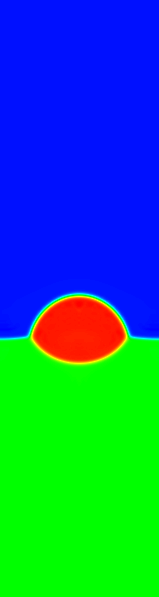}
\caption{$t=0.25$}
\end{subfigure}
\begin{subfigure}{0.12\textwidth}
\centering
\includegraphics[width=1\textwidth]{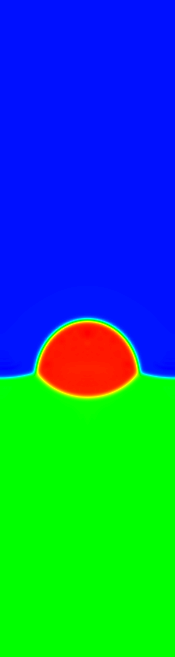}
\caption{$t=0.30$}
\end{subfigure}
\begin{subfigure}{0.12\textwidth}
\centering
\includegraphics[width=1\textwidth]{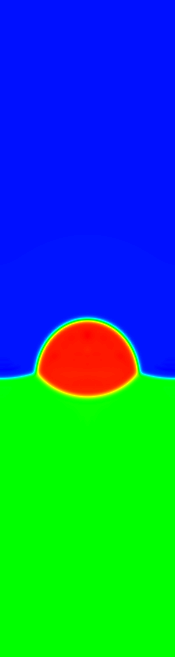}
\caption{$t=0.60$}
\end{subfigure}
\begin{subfigure}{0.12\textwidth}
\centering
\includegraphics[width=1\textwidth]{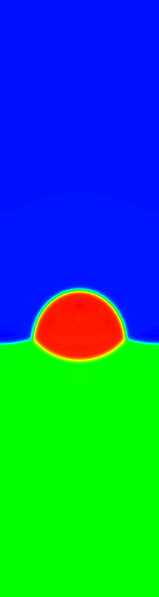}
\caption{$t=0.80$}
\end{subfigure}
\begin{subfigure}{0.12\textwidth}
\centering
\includegraphics[width=1\textwidth]{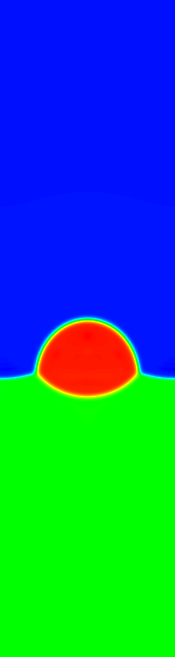}
\caption{$t=1.00$}
\end{subfigure}
\end{center}
\caption{Three-fluid problem. Evolution of the quantity $0.5\phi_1+\phi_2$ for $r=0.0015~{\rm m}$.}
\label{fig: 3 fluid - case 1 phi}
\end{figure}

\begin{figure}[!ht]
\captionsetup[subfigure]{justification=centering}
\begin{subfigure}{0.12\textwidth}
\centering
\includegraphics[width=1\textwidth]{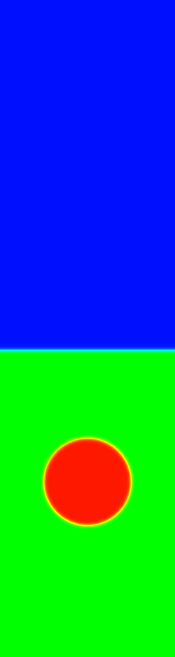}
\caption{$t=0.00$}
\end{subfigure}
\begin{subfigure}{0.12\textwidth}
\centering
\includegraphics[width=1\textwidth]{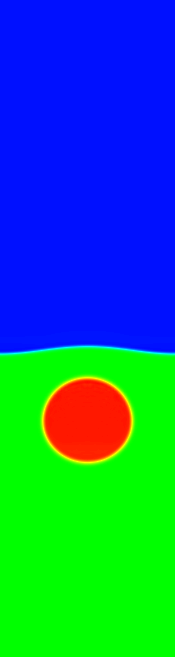}
\caption{$t=0.10$}
\end{subfigure}
\begin{subfigure}{0.12\textwidth}
\centering
\includegraphics[width=1\textwidth]{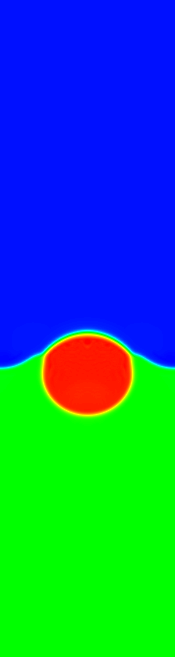}
\caption{$t=0.20$}
\end{subfigure}
\begin{subfigure}{0.12\textwidth}
\centering
\includegraphics[width=1\textwidth]{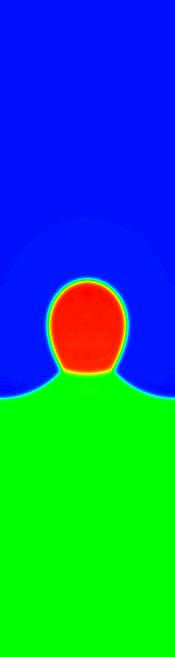}
\caption{$t=0.60$}
\end{subfigure}
\begin{subfigure}{0.12\textwidth}
\centering
\includegraphics[width=1\textwidth]{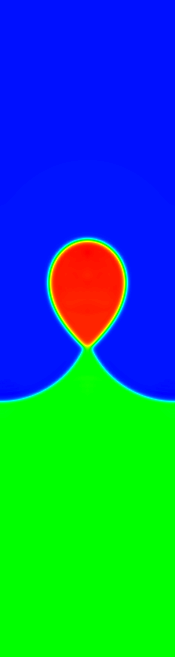}
\caption{$t=1.05$}
\end{subfigure}
\begin{subfigure}{0.12\textwidth}
\centering
\includegraphics[width=1\textwidth]{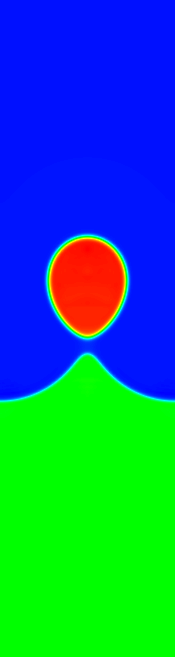}
\caption{$t=1.07$}
\end{subfigure}
\begin{subfigure}{0.12\textwidth}
\centering
\includegraphics[width=1\textwidth]{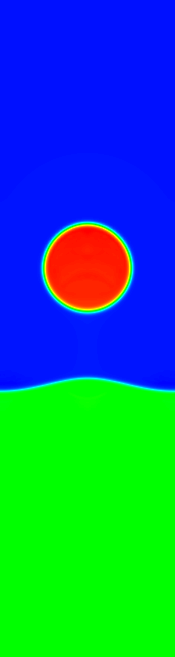}
\caption{$t=1.10$}
\end{subfigure}
\begin{subfigure}{0.12\textwidth}
\centering
\includegraphics[width=1\textwidth]{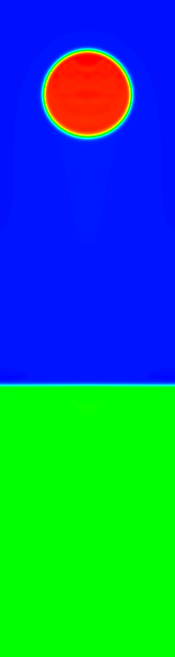}
\caption{$t=1.40$}
\end{subfigure}
\caption{Three-fluid problem. Evolution of the quantity $0.5\phi_1+\phi_2$ for $r=0.0020~{\rm m}$.}
\label{fig: 3 fluid - case 2 phi}
\end{figure}

The computed evolutions are shown in \cref{fig: 3 fluid - case 1 phi,fig: 3 fluid - case 2 phi}. For $r_1=1.5\times10^{-3}\,{\rm m}$, the bubble rises toward the liquid--liquid interface, deforms it, and remains trapped. For $r_2=2.0\times10^{-3}\,{\rm m}$, the bubble overcomes the capillary barrier and passes into the upper liquid layer. The observed transition is therefore
consistent with the two-dimensional criterion in
\eqref{eq:rp_criterion}.

\section{Conclusions}\label{sec:conclusion and outlook}
We have developed a structure-preserving fully-discrete method for an $N$-phase incompressible Navier--Stokes--Cahn--Hilliard model. Starting from the continuum formulation, we first provided the conservation and dissipation properties that should be retained at the discrete level. We then introduced an equivalent reformulation that is suitable for time discretization and finite element approximation, while keeping the treatment of the $N$ phases symmetric.

The resulting scheme preserves the main structural properties of the model. In particular, the fully-discrete formulation maintains the saturation constraint and the phase mass balances, and it satisfies a discrete energy-dissipation law. The numerical experiments confirm the theoretical properties of the method: the discrete energy decays consistently, the relevant conserved quantities are maintained up to solver tolerance, and the method produces stable evolutions for multiphase configurations with
density and viscosity differences.

Several directions remain open for future work. First, it would be valuable to extend the analysis to schemes that also preserve the mixture-aware structural properties of more general $N$-phase closures, including the corresponding mobility and reduction-consistency properties. Second, higher-order temporal discretizations are of interest. A main challenge in this direction is to retain the pointwise saturation constraint at the fully-discrete level while preserving the energy-dissipation structure. Third, the ideas developed here could be extended to compressible $N$-phase flows, for example for the model proposed in
\cite{ten2026compressible}.

\appendix

\section{Proof of the energy evolution}\label{appendix: energy}

\begin{lemma}[Energy evolution]\label[Lemma]{appendix: lem: alternative energy evolution}
The energy evolution of the system \eqref{eq:sys2} follows from a linear combination of \eqref{eq:sys2: mom}-\eqref{eq:sys2: chem} with the weights: $\vv$, $-\frac{1}{2}|\vv|^2 \rho_\mA + \mu_\mA + g y \rho_\mA$, $\lambda$, and $-\partial_t \phi_\mA$.
\end{lemma}
\begin{proof}
Taking the inner product of the momentum equation \eqref{eq:sys2: mom} with $\vv$ yields:
  \begin{align}\label{eq:proof4: mom}
    0 &= \partial_t K + \div (K \vv) + \frac{1}{2}|\vv|^2 (\partial_t \rho+\div(\rho \vv)) - \vv \cdot \div \mathbf{S}  + \displaystyle\sum_\mA\phi_\mA \vv\cdot \nabla \mu_\mA + \vv\cdot \nabla \lambda - \rho\vv\cdot\mathbf{g},
  \end{align}
  where we have used the identity:
  \begin{subequations}
    \begin{align}
     \partial_t K + \div (K \vv) + \frac{1}{2}|\vv|^2 (\partial_t \rho+\div(\rho \vv)) =&~ \vv\cdot \left( \partial_t (\rho \vv) +  \div(\rho \vv \otimes \vv)\right).
    \end{align}
  \end{subequations}
  Next, multiplying \eqref{eq:sys2: mass} with $\mu_\mA+\rho_\mA g y-\frac{1}{2}|\vv|^2 \rho_\mA $ and \eqref{eq:sys2: chem} with $-\partial_t \phi_\mA$, and subsequently summing over $\mA$ yields:
  \begin{align}\label{eq:proof4: phase+chem}
      0 = &~\partial_t (G + \Psi) + \sum_\mA (\mu_\mA+gy \rho_\mA) \div(\phi_\mA\vv) + \sum_\mA \mu_\mA \rho_\mA^{-1} \div \bJ_\mA \nn\\
      &~- \div \left(\sum_\mA \partial_t \phi_\mA \dfrac{\partial \Psi}{\partial \nabla \phi_\mA} \right)- \frac{1}{2}|\vv|^2 (\partial_t \rho+\div(\rho \vv)), %+ \lambda \partial_t \sum_\mA  \phi_\mA,
  \end{align}
  where we have used the identities:
  \begin{subequations}
      \begin{align}
      \partial_t G = &~ gy \sum_\mA \rho_\mA \partial_t \phi_\mA, \\
      \partial_t \Psi - \div \left( \partial_t \phi_\mA \dfrac{\partial \Psi}{\partial \nabla \phi_\mA} \right) =&~ \sum_\mA\partial_t \phi_\mA \dfrac{\partial \Psi}{\partial \phi_\mA} - \sum_\mA\partial_t \phi_\mA \div \left(\dfrac{\partial \Psi}{\partial \nabla \phi_\mA}\right) ,\\
      \sum_\mA (gy - \frac{1}{2} |\vv|^2 ) \div \bJ_\mA  =&~0,\\
      - \frac{1}{2}|\vv|^2 (\partial_t \rho+\div(\rho \vv)) =&~ -\sum_\mA \rho_\mA \frac{1}{2}|\vv|^2 (\partial_t \phi_\mA +\div(\phi_\mA \vv)).
  \end{align}
  \end{subequations}
  Multiplying \eqref{eq:sys2: div} with $\lambda$ provides:
  \begin{align}\label{eq:proof4: div}
 0 = \lambda\div \vv  + \lambda\displaystyle\sum_\mA \rho_\mA^{-1} \div \bJ_\mA.
 \end{align}
Addition of \eqref{eq:proof4: mom}, \eqref{eq:proof4: div},  and \eqref{eq:proof4: phase+chem} gives:
\begin{align}
  \partial_t (K + G + \Psi)  + \div ((K+\Psi + G) \vv) - \div (\mathbf{S}\vv)  +  \div ( \vv^T \mathbf{K} ) & \nn\\
  + \div( \vv \lambda )  + \div (\sum_\mA (\mu_\mA+\lambda) \rho^{-1}_\mA  \bJ_\mA)- \div \left(\sum_\mA \dot{\phi}_\mA \dfrac{\partial \Psi}{\partial \nabla \phi_\mA} \right) &~= -\mathbf{S}:\nabla \vv\nn\\
  &~~~~+ \sum_\mA \rho^{-1}_\mA \nabla (\mu_\mA+\lambda) \cdot \bJ_\mA,
\end{align}
  where we have used the identities:
  \begin{subequations}
      \begin{align}
     \div (G \vv) = &~- \rho\vv\cdot\mathbf{g} +  \sum_\mA (gy \rho_\mA) \div(\phi_\mA\vv),\\
      \sum_\mA \rho^{-1}_\mA(\mu_\mA+\lambda) \div \bJ_\mA = &~ - \sum_\mA \rho^{-1}_\mA\nabla (\mu_\mA+\lambda) \cdot \bJ_\mA + \sum_\mA \rho^{-1}_\mA\div ((\mu_\mA+\lambda) \bJ_\mA),\\
      \sum_\mA \div (\mu_\mA \phi_\mA \vv) = &~ \sum_\mA \phi_\mA \vv\cdot \nabla \mu_\mA + \mu_\mA \div (\phi_\mA \vv),\\
      \div( \vv \lambda ) =&~  \vv \cdot \nabla \lambda + \lambda\div(\vv),\\
      - \div \left(\sum_\mA \partial_t \phi_\mA \dfrac{\partial \Psi}{\partial \nabla \phi_\mA} \right) =&~ - \div \left(\sum_\mA \dot{\phi}_\mA \dfrac{\partial \Psi}{\partial \nabla \phi_\mA} \right) + \div \left( \vv^T \sum_\mA \nabla \phi_\mA \otimes \dfrac{\partial \Psi}{\partial \nabla \phi_\mA} \right),
  \end{align}
  \end{subequations}
and where the Korteweg tensor $\mathbf{K}$ is given by:
 \begin{align}
     \mathbf{K} = (\displaystyle\sum_\mA (\phi_\mA\mu_\mA)-\Psi)\mathbf{I} +   \sum_\mA \nabla \phi_\mA \otimes \dfrac{\partial \Psi}{\partial \nabla \phi_\mA} .
 \end{align}
 Integration over $\Omega$ provides:
 \begin{align}
   \dfrac{{\rm d}}{{\rm d}t}\mathcal{E} = -\mathcal{D}(\bv,\left\{g_\mA\right\} ) + \mathcal{B},
 \end{align}
 where $\mathcal{B}$ is the boundary contribution:
 \begin{align}
   \mathcal{B} =&~ - \la (K+G+\Psi) \vv, \mathbf{n} \ra_{\partial\Omega} + \la \mathbf{S} \vv, \mathbf{n} \ra_{\partial\Omega} - \la \vv^T\mathbf{K}, \mathbf{n} \ra_{\partial\Omega} \nn\\
   &~- \la \lambda\vv,\mathbf n\ra_{\partial\Omega}  - \la \sum_\mA g_\mA \bJ_\mA , \mathbf{n} \ra_{\partial\Omega} + \la \sum_\mA \dot{\phi}_\mA \dfrac{\partial \Psi}{\partial \nabla \phi_\mA}, \mathbf{n} \ra_{\partial\Omega}, 
 \end{align}
 and where we note the identities: 
  \begin{subequations}
      \begin{align}
        \mathbf{S}:\nabla \vv=&~ 2 \nu \left( \nabla^s \bv - \frac{1}{d} ({\rm div} \mathbf{v}) \mathbf{I}\right):\left(\nabla^s \bv - \frac{1}{d} ({\rm div} \mathbf{v}) \mathbf{I}\right)+ \nu \left(\bar{\lambda} + \frac{2}{d}\right)\left({\rm div} \mathbf{v}\right)^2 \geq 0,\\
        -\sum_\mA \nabla g_\mA \cdot \bJ_\mA =&~ \displaystyle\sum_{\mA,\mB} \left(\nabla g_\mA\right)^T M_{\mA\mB} \nabla g_\mB \geq 0.
  \end{align}
  \end{subequations}
\end{proof}

%\noindent {\small \textbf{Funding.} 
\subsection*{Funding}
MtE acknowledges support from the German Research Foundation (Deutsche Forschungsgemeinschaft DFG), project number 566600860. A.B. acknowledges support by the German Research Foundation (DFG) via TRR 146, subproject C3, project number 233630050 and via SPP 2256 within the project ``Variational quantitative phase-field modeling and simulation of powder bed fusion additive manufacturing'' project number 441153493. MtE and AB acknowledges support by the RMU via the joined project "MULTIMIX".

\bibliography{jfm}

\end{document}